\documentclass[11pt]{amsart}
\usepackage[latin1]{inputenc}
\usepackage{amsmath}
\usepackage{amssymb}
\usepackage{amsfonts}
\usepackage{amsthm}
\usepackage{url}
\usepackage{units}
\usepackage[a4paper]{geometry}

\usepackage[numbers]{natbib}

\usepackage{color}
\usepackage{hyperref}

\usepackage{xcolor}
\hypersetup{
    colorlinks,
    linkcolor={red!50!black},
    citecolor={blue!50!black},
    urlcolor={blue!80!black}
}

\usepackage[capitalise]{cleveref}

\newcommand{\E}{\ensuremath{\mathbb{E}}}

\newcommand{\Prob}{\ensuremath{\mathbb{P}}}
\newcommand{\C}{\ensuremath{\mathbb{C}}}
\newcommand{\R}{\ensuremath{\mathbb{R}}}
\newcommand{\N}{\ensuremath{\mathbb{N}}}
\newcommand{\Z}{\ensuremath{\mathbb{Z}}}
\newcommand{\bo}{\ensuremath{\mathrm{O}}}
\newcommand{\Ind}{\ensuremath{\textbf{1}}}

\DeclareMathOperator*{\argmax}{arg\,max}

\newcommand{\asec}{\ensuremath{\alpha^{\mathrm{sec}}}}
\newcommand{\amax}{\ensuremath{\alpha^{\mathrm{max}}}}

\newcommand{\Supp}{\ensuremath{\mathrm{Supp}}}

\renewcommand{\Re}{\ensuremath{\mathfrak{Re}}}
\renewcommand{\Im}{\ensuremath{\mathfrak{Im}}}

\newtheorem{thm}{Theorem}[section]
\newtheorem{prop}[thm]{Proposition}

\newtheorem{lem}[thm]{Lemma}
\newtheorem{rem}[thm]{Remark}

\theoremstyle{definition}
\newtheorem{defn}[thm]{Definition}

\begin{document}
\title[Densities for Stochastic Fixed Points]{On Densities for Solutions to Stochastic Fixed Point Equations}

\author{Kevin Leckey}
\thanks{Research fellow at TU Dortmund, Germany. Contact: kevin.leckey@tu-dortmund.de. This article was written during a research fellowship at Monash University, Australia.}

\begin{abstract}
We consider systems of stochastic fixed point equations that arise in the asymptotic analysis of random recursive structures and algorithms such as Quicksort, large P{\'o}lya urn processes, and path lengths of random recursive trees and split trees. 
The main result states sufficient conditions on the fixed point equations that imply the existence of bounded, smooth, rapidly decreasing Lebesgue densities.
\end{abstract}

\maketitle



\section{Introduction}

The contraction method is an approach to derive limit theorems for a broad class of random recursive structures and algorithms. It was introduced by R{\"o}sler \citep{Roe91} in 1991 for the distributional analysis of the complexity of  Quicksort. 
Over the last 25 years this approach has been extended to a variety of random variables with underlying recursive structures. Some examples are:  recursive algorithms \citep{Roe01,Ne01,NeRue04,RaRue95}, data structures \citep{LNS12,NeRue04,NeSu15},  P{\'o}lya urn models \citep{KnNe14,MueNe15}, and random tree models \citep{BrouHolm,LeNe13}.

Limit distributions derived by the contraction method are given implicitly as solutions to stochastic fixed point equations. A stochastic fixed point equation is an equation $\mu=T(\mu)$, where $\mu\in\mathcal{M}$ and $T:\mathcal{M}\rightarrow\mathcal{M}$  for some set $\mathcal{M}$ of probability distributions. A random variable $X$ is called solution to the stochastic fixed point
equation $\mu=T(\mu)$ if its distribution $\mu$ is a fixed point of $T$. In many cases an explicit description of such a solution (e.g.~in terms of its distribution function) is unknown. 
In particular a lot of limits in P\'olya urn models are not known explicitly and thus 
any further properties of these limits need to be derived from their stochastic fixed point equations.

The aim of this paper is a better understanding of solutions to stochastic fixed point equations. 
We will discuss several examples in which the fixed point equation implies that the limit distribution is 'smooth' in the sense that it has an infinitely differentiable, rapidly decreasing Lebesgue density. 

The stochastic fixed point equations in this paper are of the following type: Let $X_1,\ldots,X_m$ be $\R^d$-valued random variables for some $m\geq 1$ and $d\geq 1$. Here and subsequently, $\R^d$ denotes the set of {\it column} vectors of dimension $d$. The distributions of $X_1,\ldots,X_m$ are given implicitly as solutions to
 \begin{align}\label{main_rec_intro}
 X_r\stackrel{d}{=} \sum_{j=1}^{\infty} A_{r,j} X_{\ell_r(j)}^{(j)}+b_r,\quad r\in [m]:=\{1,\ldots,m\},
\end{align}
where $\stackrel{d}{=}$ denotes equality in distribution, and:  
\begin{itemize}
\item $\ell_r:\N\rightarrow [m]$ is some given function;
\item $X_{\ell_r(j)}^{(j)}$ has the same distribution as $X_{\ell_r(j)}$;
\item $A_{r,j}$ is a random $d\times d$ matrix and $b_r$ is a $\R^d$-valued random variable;
\item $X_{\ell_r(1)}^{(1)},X_{\ell_r(2)}^{(2)},\ldots$ and $((A_{r,j})_{j\geq 1},b_r)$ are independent.
\end{itemize}
The infinite sum in \eqref{main_rec_intro} denotes the a.s.~limit of the partial sums, thus we assume implicitly that the sequences
\begin{align*}
\left(\sum_{j=1}^{n} A_{r,j} X_{\ell_r(j)}^{(j)}\right)_{n\geq 1}
\end{align*}
converge almost surely as $n\rightarrow\infty$ for every $r\in[m]$. 

The main result (\cref{thm_A_conds}) states sufficient conditions on $((A_{r,j})_{j\geq 1},b_r)$ for the existence of bounded (Lebesgue-) densities for the distributions of $X_1,\ldots,X_m$. These densities are shown to be smooth (i.e.~all derivatives exist). Moreover, they are Schwartz functions if all moments of $X_1,\ldots,X_m$ are finite ($f$ is a Schwartz function if $f(t)$ and all its derivatives decay faster than any polynomial in $\|t\|^{-1}$). 

The methods in this article are based on known results for branching processes \cite{liu99}, Quicksort \cite{fiJa00} and P\'olya urns \cite{ChMaPou13,Ma15}. Aside from improving some of these results, we manage to provide a general framework that covers other examples such as limit laws in several random tree models and multivariate limit laws. In particular, most of the previous results only studied one dimensional ($d=1$) cases with one equation ($m=1$). The only notable exception are the results on P\'olya urns \cite{ChMaPou13,Ma15}, which dealt with more than one equation of possibly complex valued random variables. However, the methods used for P\'olya urns \cite{ChMaPou13,Ma15} were not as powerful as the ones for Quicksort \cite{fiJa00} in the sense that they could not prove that the limit densities are Schwartz functions.

There are several reasons to derive properties of distributions from their fixed point equations. First of all we want to provide a better understanding of stochastic fixed point equations in general since
they appear naturally in various areas such as recursive algorithms, random trees and P{\'o}lya urns (more details are given in the next paragraph).  The second reason is to reduce redundancy in future works: The existence (and sometimes smoothness) of densities has been shown for some examples \cite{liu99,fiJa00,ChMaPou13,Ma15} using similar methods. This article not only merges those approaches but also covers a variety of other limits, some of them presented in \cref{sec:motivation}.
The last reason is connected to the contraction method itself\footnote{A prominent example of a limit theorem derived by this method is Quicksort \cite{Roe91}}. In the one dimensional case ($d=1$) the existence of a bounded density is useful to derive rates of convergence:
The contraction method deals with distances 
in a rather abstract metric, often in the so-called Wasserstein metric $\ell_p$ (cf., e.g., \cite{Roe91}).
Usually it is not hard to extract a rate of convergence in this abstract metric, that is an upper bounds on $\ell_p(Y_n,Y)$ for a sequence $(Y_n)_{n\geq 1}$ that converges to $Y$. 
Now suppose $Y_n$ and $Y$ are real valued and let $F_n$ denote the distributional function of $Y_n$ and let $F$ denote the distributional function of $Y$. Assume that $Y$ admits a bounded density $f_Y$. Then one can show that
\begin{align*}
\|F_n-F\|_\infty \leq \left( (p+1)\| f_Y\|_\infty^p\right)^{\frac 1 {1+p}} \left(\ell_p(Y_n,Y)\right)^{\frac p {1+p}},\qquad p\geq 1.
\end{align*}
Thus, the existence of a density for $Y$ is crucial to transfer a rate of convergence of $(\ell_p(Y_n,Y))_{n\geq 1}$ to an upper bound on $(\|F_n-F\|_\infty)_{n\geq 1}$.

We continue the introduction with some examples where equations like \eqref{main_rec_intro} appear. More details on these examples are given in \cref{sec:motivation}. 
\begin{itemize}
\item{ {\bf Branching processes.} Limit distributions in several branching processes \citep{liu01} can be 
characterized by an equation of type \eqref{main_rec_intro} with: 
\begin{align*}
m=1,\quad d=1,\quad (A_{1,j})_{j\geq 1}= ( A_{1,j}^\prime \Ind_{\{N\geq j\}})_{j\geq 1},\quad b_1=0,
\end{align*}
where $(N, A_{1,1}^\prime,A_{1,2}^\prime,\ldots)$ is a random variable in $\N_0\times (0,\infty)\times (0,\infty)\times\ldots$ 

In a supercritical Galton Watson process $(Z_n)_{n\geq 0}$, for example, let $N$ be the offspring distribution and $\mu=\E[N]$. Then $(Z_n/\mu^n)_{n\geq 0}$ converges to a limit that solves this kind of distributional equation with $A_{1,j}^\prime=1/\mu$. Note that the limit distribution has an atom in $0$ unless $\Prob(N=0)=0$. Hence it cannot be absolutely continuous on its entire support.

If $\Prob(N=0)=0$, then Liu \cite{liu01} states sufficient conditions on $(N, A_{1,1}^\prime,A_{1,2}^\prime,\ldots)$ for the existence of a density for $X_1$.  
We extend this approach to obtain similar results for the general equation \eqref{main_rec_intro} (see \cref{lem_conv_zero} and \cref{thm_poly_decay}). Although these result are insufficient to obtain smooth densities for the other examples listed below, they provide a basis for inductively gaining better bounds on the characteristic functions of $X_1,\ldots,X_m$. These bounds finally lead to a result (\cref{thm_A_conds}) that implies the existence of infinitely differentiable densities for all examples below.

We will not discuss branching processes in this paper although it improves Liu's result \cite[Corollary (Absolute continuity I)]{liu01} in some situations where the largest element
among $\{|A_{1,1}^\prime|,\ldots, | A_{1,N}^\prime| \}$ is bounded from below but
\begin{align}\label{intro_liu_b}
\E\left[\sum_{j=1}^N |A_{1,j}^\prime|^{-b}\right]=\infty\quad\text{ for some }b>0.
\end{align}
Note that with \eqref{intro_liu_b} Liu's result may still yield the existence of a density $f$, but 
cannot ensure that the $a$-th derivative of $f$ exists for $a>b$. The result in this paper uses a different approach that yields the existence of all derivatives of $f$ (under conditions introduced in \cref{def_A_conds}). 
}
\item{ {\bf Quicksort.} The Quicksort limit $X$ \citep{Roe91} satisfies an equation of type \eqref{main_rec_intro} with 
\begin{align*}
m=1,\quad d=1,\quad (A_{1,j})_{j\geq 1}=(U,1-U,0,\ldots),\quad b_1=g(U),
\end{align*}
 where $U$ is uniformly distributed on $[0,1]$ and $g:[0,1]\rightarrow \R$ is some function. Based on this equation Fill and Janson \cite{fiJa00} show that $X$ has a density which is a Schwartz function. Parts of the proofs in Section \ref{sec_proofs} (\cref{lem_c7_p} and \cref{lem_Schwartz_p}) are largely inspired by the work of Fill and Janson.

Note that the joint distribution of key comparisons and key exchanges in Quicksort also converges to a limit that can be described by an equation of type \eqref{main_rec_intro} (see \cite{Ne01}).
\cref{thm_A_conds} yields that this joint limit (as a random variable in $\R^2$) has an infinitely differentiable density. To the best of this author's knowledge this is the first\footnote{aside from $\C$-valued limits as a special case of $d=2$; cf., e.g., \cite{CGPT16,ChLiPou12,Ma15}} result on densities for solutions to \eqref{main_rec_intro} in $\R^d$ for $d\geq 2$.
 }
\item{ {\bf P{\'o}lya urns.} Consider a P{\'o}lya urn with $q$ colors and some replacement rule $R=(\xi_{i,j})_{i,j\in [q]}\in\N_0^{q\times q}$ (cf. Section \ref{sec:motivation} for details). 
A lot of replacement rules (often called large P\'olya urns) are known to lead to non-normal limit laws; see Janson \cite{Ja04}.
 Usually limit distributions in this context can be characterized by equations of type \eqref{main_rec_intro}; see \cite{KnNe14} for some examples. These cases are in general multidimensional ($m\geq 2$), which make them fall out of Liu's framework \cite{liu01}. 
 
It is known that in the case of two colors these limits have smooth
densities \cite{ChMaPou13}. For more than two colors, similar arguments show that at least the projections of the limits to the eigenspaces of $R$ have densities \cite{Ma15}.
\cref{thm_A_conds} provides a general framework for these examples (and others such as random replacement rules) and yields the existence of infinitely differentiable density functions for the limits.

Note that similar distributional equations appear in the context of $B$-urns \cite{CGPT16} and $m$-ary search trees \cite{ChLiPou14}  since the analysis of both processes is closely related to P{\'o}lya urns.

Finally note that the P{\'o}lya urn process has a well-known continuous time embedding (related to multitype branching processes; cf., e.g., \cite{Ja04} for P{\'o}lya urns or \cite{ChLiPou14} for $m$-ary search trees ). Limits of these continuous time counterparts can also be characterized by equations of type \eqref{main_rec_intro}; see \cite{ChMaPou13,ChPouSah11,Ma15} for some examples. In some cases it is known that the limits have densities \cite{ChPouSah11}, which, however, explode at $0$ \cite[Proposition 4.2]{ChPouSah11}. Since the methods in this paper are based on Fourier transformation (leading to continuous densities), we cannot hope to extend our methods to 
 continuous time models.
}
\item{ {\bf Path length of random trees.} The (total) path lengths of several random trees, in particular the class of random split trees \cite{Dev98}, converge to limits that can be described by equations like \eqref{main_rec_intro}. To the best of this author's knowledge the limit distribution in split trees has not been analyzed regarding absolute continuity. Under some mild assumption on the split vector (which are fulfilled in all examples given by Devroye \cite{Dev98}) \cref{thm_A_conds} yields that these limits have infinitely differentiable densities.

Another advantage of the general setting in \eqref{main_rec_intro} is that it also covers multivariate limits: The joint distribution of path length and Wiener index in a split tree converges to a limit distribution \cite{Mun11} given by an equation of type \eqref{main_rec_intro} (in $\R^2$). Again \cref{thm_A_conds} implies that the limit distribution has an infinitely differentiable density (in $\R^2$).}
\end{itemize}
To wrap up the advantages of this paper: One of the main advantages is the generality of the framework. Aside from Liu's result \cite{liu01}, all known results listed above are derived for an explicit type of $(A_{r,j})_{j\geq 1}$, that is either a beta- or a Dirichlet-distribution. This paper does not require such restrictive assumptions on $(A_{r,j})_{j\geq 1}$. Another advantage is that, despite its generality, the result is as powerful as the corresponding one for Quicksort \cite{fiJa00}, which in many examples beats the approach by Liu. Moreover, it covers systems of fixed point equations ($m\geq 2$), which arise, e.g., in P{\'o}lya urns. Furthermore it is the first result on fixed point equations for random variables in higher dimensions ($d\geq 2$) that arise in multivariate limit laws. 
We will also see in \cref{sec:motivation} how the main result can be applied straightforwardly to diverse examples.

The paper is organized as follows. Section \ref{sec:def} contains some basic notation 
and definitions. In particular, Schwartz functions are defined in that section. The main result (\cref{thm_A_conds}) is presented in \cref{sec_main_thm}. \cref{sec_main_thm} also contains some results that are proven as a preparation for \cref{thm_A_conds}. 
\cref{sec:motivation} provides several examples that are covered by \cref{thm_A_conds} such as Quicksort, P{\'o}lya urns, and path lengths in several random tree models.  Full proofs for the results in \cref{sec_main_thm} are given in \cref{sec_proofs}.
Finally, \cref{sec:conclusion} contains a discussion on the assumptions made in \cref{thm_A_conds} and some further directions for the analysis of stochastic fixed point equations.

\section{Preliminaries and Notation}\label{sec:def}

Throughout this paper let $d$ be a positive integer. Let $\R^d$ be endowed with the standard inner product $\langle x,y\rangle$ and Euclidean norm $\|x\|$ of vectors $x,y\in\R^d$.
The operator norm of a matrix $A$ is denoted by $\|A\|_{\mathrm{op}}$.
Let $\Re(z)$ and $\Im(z)$ denote real- and imaginary part of a complex number $z$.
Complex numbers are embedded into $\R^2$ as usual, that is $z\in \C$ is identified with the vector $(\Re(z),\Im(z))\in\R^2$. In particular, $\langle z_1,z_2\rangle:=\Re(z_1)\Re(z_2)+\Im(z_1)\Im(z_2)$ for $z_1,z_2\in\C$.\\

{\bf Notation.} Let $\N$ denote the set of all positive integers and let $\N_0:=\N\cup \{0\}$. Let $[m]:=\{1,\ldots,m\}$ for a positive integer $m$. Let $a\wedge b:=\min(a,b)$ and $a\vee b:=\max(a,b)$ for real numbers $a$ and $b$.  Let $\mathcal{L}(X)$ denote the distribution of a random variable $X$. 

Depending on the context, let ${\bf 0}$ either denote the zero vector of $\R^d$ or the $d\times d$ zero matrix, $d\geq 2$.

For a function $f:\R^d\rightarrow\R$ and a vector $\beta=(\beta_1,\ldots,\beta_d)\in\N_0^d$ let 
\begin{align*}
D^{\beta} f:=\frac{\partial^{\beta_1}}{\partial x_1^{\beta_1}}\cdots \frac{\partial^{\beta_d}}{\partial x_d^{\beta_d}} f.
\end{align*}
A function $f$ is called $n$-times continuously differentiable if $D^\beta f$ exists and is continuous for all $\beta$ with $\sum_j \beta_j\leq n$.
 Let $\mathcal{C}^n(\R^d)$ be the set of all $n$-times continuously differentiable functions $f:\R^d\rightarrow\R$. 
Moreover, let $\mathcal{C}^\infty(\R^d):=\bigcap_{n\in\N} \mathcal{C}^n(\R^d)$. 
\begin{defn}\label{def_Schwartz} 
 A function $f:\R^d\rightarrow \R$ is called a \emph{Schwartz function}
if 
\begin{align*}
(i)\;f\in\mathcal{C}^\infty(\R^d),\qquad (ii)\; \sup_{x\in\R^d}\left\{ \|x\|^\alpha \left| D^\beta f(x)\right|\right\} <\infty\text{ for all $\alpha\in\N$ and $\beta\in\N_0^d$}.
\end{align*}
Moreover, a function $g:\C\rightarrow \R$ is called a \emph{Schwartz function} if the function $f:\R^2\rightarrow\R$, $(x_1,x_2)\mapsto g(x_1+i x_2)$, is a Schwartz function.
\end{defn}
\begin{defn}\label{defSchwartzdens}
A probability distribution $\mu$ on $\R^d$ has a \emph{Schwartz density} if and only if there is a Schwartz function $f$ such that $f$ is a density of $\mu$.
An $\R^d$-valued random variable $X$ \emph{admits a Schwartz density} if and only if $\mathcal{L}(X)$ has a Schwartz density.
\end{defn}
\begin{defn}\label{defSuppGenPos}
$\Supp(X)$ denotes the support of a $\R^d$-valued random variable $X$, i.e.
\begin{align*}
\Supp(X):=\{x\in\R^d : \Prob(\|X-x\|<\varepsilon)>0\text{ for all }\varepsilon>0\}.
\end{align*}
$\Supp(X)$ is in {\it general position} if there are $x_0,\ldots,x_d\in\Supp(X)$ such that 
\begin{align*}
x_1-x_0,x_2-x_0,\ldots,x_d-x_0\text{ are linearly independent in $\R^d$.}
\end{align*}
For $d=1$, $\Supp(X)$ is in general position if and only if $\Prob(X=c)<1$ for all $c\in\R$.
\end{defn}
\begin{defn}\label{def_nonlattice}
	An $\R^d$-valued random variable $X$ has a non-lattice distribution if
	\begin{align*}
	\Prob(\langle s, X \rangle \in \Z+c)<1\text{ for all }s\in\R^d\setminus\{{\bf 0}\}\text{ and }c\in\R,
	\end{align*}
	where ${\bf 0}$ denotes the zero vector of $\R^d$.
\end{defn}

\section{Main Results}\label{sec_main_thm}

Let $m$ be a positive integer. Let $X_1,\ldots,X_m$ be $\R^d$-valued random variables which solve a system of distributional equations introduced in Equation \eqref{main_rec} below.
Assume for every $r\in[m]$ the existence of a function $\ell_r:\N\rightarrow[m]$, a family $(A_{r,j})_{j\geq 1}$ of random $d\times d$ matrices, and a $\R^d$-valued random variable $b_r$ such that $X_r$ satisfies the following distributional equation:
 \begin{align}\label{main_rec}\tag{DE}
 X_r\stackrel{d}{=} \sum_{j=1}^{\infty} A_{r,j} X_{\ell_r(j)}^{(j)}+b_r,\quad r\in [m],
\end{align}
where $X_{\ell_r(1)}^{(1)},X_{\ell_r(2)}^{(2)},\ldots$ and $((A_{r,j})_{j\geq 1},b_r)$ are independent, and $X_{\ell_r(j)}^{(j)}$ has the same distribution as $X_{\ell_r(j)}$ for all $j\geq 1$.
The infinite series in \eqref{main_rec} denote the a.s.~limit of the partial sums, i.e.~we assume implicitly at this point that the partial sums converge almost surely.
To ensure this convergence, we make the following assumption throughout the article: There is a constant $\varepsilon>0$ such that {\bf at least one} of the following two conditions holds for all $r\in [m]$:
\begin{flalign}
&\#\{ j: A_{r,j}\neq {\bf 0}\}<\infty\quad \text{a.s.},\tag{S.a}\label{cond_Sa}\\
 &\E[\|X_r\|^{\varepsilon}]<\infty\quad \text{and}\quad \limsup_{j\rightarrow\infty}(j\log^2 j)^{1+1/\varepsilon}\|A_{r,j}\|_{\mathrm{op}}<\infty \quad\text{a.s.}\tag{S.b}\label{cond_Sb}.
\end{flalign}
\begin{rem}In order to see that \eqref{cond_Sb} implies the convergence of the partial sums, note that $\E[\|X_r\|^\varepsilon]<\infty$ implies
$\Prob(\|X_r\|>(j \log^2 j)^{1/\varepsilon})=\bo(1/(j\log^2 j))$ by Markov's inequality. Hence, by the Borel-Cantelli Lemma and the assumption on $A_{r,j}$,
\begin{align*}
\|A_{r,j}X_{\ell_r(j)}^{(j)}\|\leq \|A_{r,j}\|_{\mathrm{op}}\|X_{\ell_r(j)}^{(j)}\|=\bo\left(1/(j\log^2 j)\right)\quad \text{a.s.}
\end{align*}
Thus $\sum_j \|A_{r,j}X_{\ell_r(j)}^{(j)}\| <\infty$ almost surely and the convergence of the series in \eqref{main_rec} follows from the triangle inequality.
\end{rem}
As a preparation for the main results we introduce some notation. Let $A_{r,j}^T$ denote the transpose of $A_{r,j}$ and
\begin{equation}\label{def_alpha_ij}
 \alpha_{r,j}:=\min_{\|t\|=1} \|A_{r,j}^T t\|,\qquad \|A_{r,j}^T\|_{\mathrm{op}}:=\max_{\|t\|=1} \|A_{r,j}^T t\|,\quad \text{ for }r\in [m],\,j\in\N.
\end{equation}
 Moreover, let $\amax_r\geq\asec_r$ denote the two largest elements in $(\alpha_{r,j})_{j\geq 1}$. Note that these elements are well defined since \eqref{cond_Sa} or \eqref{cond_Sb} imply 
 \begin{align*}
 \alpha_{r,j}\leq\|A_{r,j}^T\|_{\mathrm{op}}=\|A_{r,j}\|_{\mathrm{op}}\longrightarrow 0\quad\text{ a.s.~as $j\rightarrow\infty$.}
 \end{align*}
 Recall that $A_{r,j}$ and thus also $\alpha_{r,j}$, $\amax_r$ and $\asec_r$ are random variables.
 
Finally, for every interval $\mathcal{I}\subset\R$ and every $r\in[m]$ let $N_r(\mathcal{I})$ be the (possibly infinite) random variable given by
\begin{equation}\label{def_nr01coeff}
N_r(\mathcal{I}):=\sum_{j=1}^{\infty} \Ind_{\{\alpha_{r,j}\in\mathcal{I}\}\cap\{\|A_{r,j}^T\|_{\mathrm{op}}\in \mathcal{I}\}}.
\end{equation}
The following conditions are tailored to the examples in the next section. More general conditions are discussed in \cref{def_c} after the main result.
\begin{defn}\label{def_A_conds}
Conditions \eqref{s_cond_2}-\eqref{s_cond_6} hold if 
for all $r\in[m]$ and $j\geq 1$:
\begin{flalign}
&\Prob(\amax_r \geq a)=1\quad \text{for some constant $a>0$,}&\label{s_cond_2}\tag{A1}\\
&\Prob(\asec_r\leq x) \leq \lambda x^{\nu}\quad \text{ for some $\lambda,\nu>0$ and all $x>0$,}&\label{s_cond_3}\tag{A2}\\
&\Prob(\|A_{r,j}^T\|_{\mathrm{op}}\leq 1)=1,&\label{s_cond_4}\tag{A3}\\
&\Supp(X_r)\text{ is in general position (see Definition \ref{defSuppGenPos}), }&\label{s_cond_5} \tag{A4}\\
&\Prob(N_r( \mathcal{I})\geq 1)>0\text{ for }\mathcal{I}:=(0,1)\subset\R.\tag{A5}&\label{s_cond_6}
\end{flalign}
\end{defn}
The main result of this paper is the following theorem:
\begin{thm}\label{thm_A_conds} Conditions \eqref{s_cond_2}-\eqref{s_cond_6} imply that
$X_r$ admits a bounded density function $f_r\in\mathcal{C}^\infty(\R^d)$ for all $r\in [m]$.
If additionally $\E[\|X_r\|^p]<\infty$ for all $r\in[m]$ and $p>0$, then $X_r$ admits a Schwartz density (see Definition \ref{defSchwartzdens}).
\end{thm}
\begin{rem}\label{rem_s_cond_5}
	Condition \eqref{s_cond_5} might not seem sensible at first since it refers to the (unknown) distributions of $X_1,\ldots,X_m$ rather than the coefficients $(A_{r,j})_{j\geq 1}$ and $b_r$. However, in several examples, e.g.~the two color urn limits discussed in \cref{sec:motivation}, \eqref{main_rec} has a degenerate solution in addition its absolute continuous solutions. Hence, only making assumptions for the coefficients is not sufficient to ensure absolute continuity of the solution.
	
	Checking \eqref{s_cond_5} for higher dimensions ($d\geq 2$) is more tedious. We will discuss two examples (bivariate Quicksort limit and the joint distribution of path length and Wiener index in split trees)
	in \cref{sec:motivation} where \eqref{s_cond_5} can be deduced from the coefficients $(A_{r,j})_{j\geq 1}$ and $b_r$.
\end{rem}

The proof of \cref{thm_A_conds} is based on Fourier analysis. 
We only outline the proof strategy in this section. Full proofs are given in \cref{sec_proofs}.

Note that the set of Schwartz functions is preserved under Fourier transformation \citep[Theorem 7.4(d)]{Ru91}. Thus the characteristic function of a distribution is a Schwartz function if and only if the distribution has a Schwartz density. For the remainder of the section we discuss conditions on \eqref{main_rec} and their effect on the characteristic functions
\begin{align*}
\phi_r(t):=\E[\exp(i\langle t, X_r\rangle)],\,t\in\R^d,\,r\in[m].
\end{align*}
 The first step of the proof is to verify that \eqref{s_cond_2}-\eqref{s_cond_6} imply the following conditions:
\begin{defn}\label{def_c} 
Conditions \eqref{c1}-\eqref{c3} hold if for all $r\in[m]$:
\begin{flalign}
&\Prob(\amax_r>0)=1 ,&\label{c1}\tag{C1}\\
&\E[N_r((0,1])]>1 , &\label{c2}\tag{C2}\\
&X_r \text{ has a non-lattice distribution (cf.~\cref{def_nonlattice}).} &\label{c3}\tag{C3}
\end{flalign}
Let $\eta>0$. Conditions \eqref{c4}-\eqref{c6} hold for $\eta$ if for all $r\in[m]$:
\begin{flalign}
&\E[(\asec_r)^{-\eta} |\asec_r>0]<\infty ,&\label{c4}\tag{C4}\\
&\Prob(\amax_r \leq x)= \bo(x^\eta)\text{ as } x\rightarrow 0,& \label{c5}\tag{C5}\\
&\E[(\amax_r)^{-\eta}\Ind_{\{\asec_r=0\}}] <1.& \label{c6}\tag{C6}
\end{flalign}
Finally, let $\chi:(0,\infty)\rightarrow(0,\infty)$ be a function. Condition \eqref{c7} holds for $\chi$ if for all $\beta>0$ a constant $C_\beta>0$ exists such that for all $x>0$ and $r\in[m]$
\begin{flalign}
& \E\left[ \prod_{j\geq 1} \left((\alpha_{r,j} x)^{-\beta} \wedge 1 \right)\right]\leq C_{\beta} x^{-\chi(\beta)}. &\label{c7}\tag{C7}
\end{flalign}
\end{defn}
Afterwards, we successively improve bounds on $|\phi_r(t)|$ as indicated in the results below. The Fourier inversion formula then yields the existence and differentiability of density functions for $X_1,\ldots,X_m$.
\begin{lem}\label{lem_conv_zero}
Assume \eqref{c1}-\eqref{c3}. Then, 
\begin{align*}
\lim_{R\rightarrow\infty} \sup_{\|t\|=R} \left| \phi_r(t)\right| =0
\text{ for all } r\in[m].
\end{align*}
\end{lem}
\begin{prop}\label{thm_poly_decay} Assume \eqref{c1}-\eqref{c3} and
\eqref{c4}-\eqref{c6} with $\eta>0$. Then,
\begin{align}\label{poly_decay_phi}
|\phi_r(t)|=\bo\left(\|t\|^{-\eta}\right)\text{ for all $r\in[m]$ as } \|t\|\rightarrow\infty.
\end{align}
If $\eta>d$, then  $X_r$ admits a bounded density function $f_r\in\mathcal{C}^{\lceil \eta \rceil -d-1}(\R^d)$ for all $r\in[m]$.
\end{prop}
Note that \cref{lem_conv_zero} and \cref{thm_poly_decay} are based on the strategy of Liu \cite{liu01}, who studied \eqref{main_rec} with $m=1$, $d=1$ and $b_1=0$. However,
 in every example in \cref{sec:motivation}, \eqref{c4} only holds for $\eta$ up to some constant $C$. Thus \cref{thm_poly_decay} is either not sufficient to prove the existence of a density function at all (if $C\leq 1$) or at least fails to prove its smoothness. 
The bound in \cref{thm_poly_decay} can often be improved by \cref{coro_smooth_density} below. To this end, let $\chi^{n}$ denote the $n$-fold composition of a function $\chi:(0,\infty)\rightarrow(0,\infty)$ with itself.
\begin{prop}\label{coro_smooth_density}
 Assume \eqref{c1}-\eqref{c6} and \eqref{c7} with $\eta>0$ and $\chi:(0,\infty)\rightarrow (0,\infty)$ such that $\limsup\limits_{n\rightarrow\infty} \chi^{n}(\eta) =\infty$. Then, for all $\beta>0$, 
\begin{align*}
|\phi_r(t)|=\bo\left( \|t\|^{-\beta}\right)\text{ for all $r\in[m]$ as } \|t\|\rightarrow\infty.
\end{align*}
In particular, $X_r$ admits a bounded density function $f_r\in\mathcal{C}^\infty(\R^d)$ for all $r\in [m]$.
\end{prop}
This results can be extended to yield a Schwartz density as follows:

\begin{lem}\label{lem_Schwartz}
Let $X$ be a $\R^d$-valued random variable with characteristic function $\phi$. 
Assume $\E[\|X\|^p]<\infty$ for all $p>0$. Moreover, assume for all $\beta>0$ that
\begin{align*}
|\phi(t)|=\bo\left( \|t\|^{-\beta}\right)\text{ as } \|t\|\rightarrow\infty.
\end{align*}
 Then $X$ admits a Schwartz density (see Definition \ref{defSchwartzdens}).
\end{lem}
\begin{rem} R{\"o}sler \citep{Roe92} analyzes stochastic fixed point equations for real valued random variables (i.e.~$d=1$ and $m=1$ in \eqref{main_rec}) regarding exponential moments. 
In particular, he states \citep[Theorem 6]{Roe92} sufficient conditions on \eqref{main_rec} for finite exponential moments.
\end{rem}

\begin{rem}\label{embed_complex} Note that the results in this section can also be applied to $\C$-valued random variables by embedding them into $\R^2$ in the canonical way.
Since some of our applications involve complex valued random variables, 
we briefly formalize this embedding. 

Let $Y_1,\ldots,Y_m$ be $\C$-valued random variables. Assume for every $r\in[m]$ the existence of 
 a function  $\ell_r:[k_r]\rightarrow\N$ and 
$\C$-valued random variables  $V_{r,1},V_{r,2},\ldots$, $B_r$ such that:
\begin{align}\label{s_main_rec}
 Y_r\stackrel{d}{=} \sum_{j=1}^{\infty} V_{r,j} Y_{\ell_r(j)}^{(j)}+B_r,
\end{align}
where $Y_{\ell_r(1)}^{(1)},Y_{\ell_r(2)}^{(2)},\ldots$ and $((V_{r,j})_{j\geq 1},B_r)$ are independent; and $Y_{\ell_r(j)}^{(j)}$ has the same distribution as $Y_{\ell_r(j)}$ for $j\geq 1$.

Then $Y_1,\ldots,Y_m$ can be embedded into \eqref{main_rec} with $d=2$ by letting
\begin{align}\label{embed_C_R2}
X_r:=\begin{pmatrix}\Re(Y_r) \\ \Im(Y_r) \end{pmatrix},\quad A_{r,j}:=\begin{pmatrix} \Re(V_{r,j}) & -\Im(V_{r,j}) \\ \Im(V_{r,j}) & \Re(V_{r,j}) \end{pmatrix},\quad b_r:=\begin{pmatrix} \Re(B_r) \\ \Im(B_r)\end{pmatrix}.
\end{align}
Note that $\| A_{r,j}^T\|_{\mathrm{op}}$ and $\alpha_{r,j}$  in \eqref{def_alpha_ij} are both equal to $|V_{r,j}|$. 
\end{rem}
\begin{rem}
	We end the section with a discussion on the assumption $\E[\|X_r\|^p]<\infty$ in \cref{thm_A_conds}. This assumption implies\footnote{More details can be found in \cref{sec_proofs}} that all derivatives of $\phi_r$ are bounded. In combination with the tail bounds in \cref{coro_smooth_density}, an argument based on Fill and Janson \cite{fiJa00} (cf.~\cref{calc_lemma}) yields that $\phi_r$ (and also $f_r$) is a Schwartz function. Now suppose we relax the moment condition, such that, for some $p_0\in\N$,
	\begin{align*}
	\E[\|X_r\|^p]<\infty\quad\text{for all }p\leq p_0\text{ and }r\in[m].
	\end{align*}
	Similar arguments as in the proof of \cref{thm_A_conds} yield (using \cref{calc_lemma}) that for all $\alpha>0$ and $\beta:=(\beta_1,\ldots,\beta_d)\in\N_0^d$ with $\sum_j \beta_j \leq p_0-1$,
	\begin{align*}
	|D^\beta \phi_r (t)|=\bo\left(\|t\|^{-\alpha}\right).
	\end{align*}
	Now recall $x^\beta=x_1^{\beta_1}\cdots x_d^{\beta_2}$ for $x\in\R^d$.
	Note that (cf., e.g., \cite[Theorem 7.4 (c)]{Ru91})
	\begin{align}\label{rem312_a}
	x^\beta f(x)=\frac 1 {(2\pi)^d} \int_{\R^d} \mathrm{e}^{-i\langle t ,x \rangle} D^\beta\phi_r(t)\mathrm{d}\lambda^d(t),
	\end{align}
	where $\lambda^d$ denotes the Lebesgue measure on $\R^d$.
	Hence 
	\begin{align*}
	|x^\beta f(x)|=\bo(1)\quad \text{for every $\beta$ with $\sum_j\beta_j\leq p_0-1$.}
	\end{align*}
	More generally, using, e.g., the Dominated Convergence Theorem and 
	 \eqref{rem312_a},
	 \begin{align*}
	 x^\beta D^\gamma f(x)=\frac 1 {(2\pi)^d} \int_{\R^d} (-i)^{\sum_j \gamma_j} t^\gamma \mathrm{e}^{-i\langle t ,x \rangle} D^\beta\phi_r(t)\mathrm{d}\lambda^d(t)\quad \text{for $\gamma\in\N_0^d$.}
	 \end{align*}
	 Thus also $|x^\beta D^\gamma f(x)|=\bo(1)$ for every $\beta,\gamma\in\N_0^d$ with $\sum_j \beta_j\leq p_0-1$ (Note that the constant in the $\bo(1)$ term may depend on $\gamma$).
\end{rem}

\section{Applications}\label{sec:motivation}

This section contains examples of limits given by stochastic fixed point equations. Checking Conditions \eqref{s_cond_2}-\eqref{s_cond_6} for real valued random variables ($d=1$) is usually straightforward and details in most examples are left to the reader.\\

\noindent
{\bf 1.~Quicksort.} The {\it Quicksort} algorithm was introduced by Hoare \citep{Hoa62} in 1962. This algorithms sorts a list by choosing a {\it pivot} element among its elements and subdividing the list into two parts: one containing the elements smaller than the pivot, the other containing the elements larger than the pivot. The algorithm then is recursively applied to both parts.
If the {\it pivot } is chosen uniformly at random in the list (or if the input is considered to be random), the total number of key comparisons, properly rescaled, converges to some limit $X$ almost surely as the number of keys tends to infinity. This convergence was first proven with martingale techniques \citep{Reg89} without specifying the limit $X$. 
With R{\"o}sler's contraction method the distribution of $X$ can be characterized as 
a solution to the following stochastic fixed point equation \citep{Roe91}:
\begin{align}\label{DE_Quicksort}
X\stackrel{d}{=} UX^{(1)}+(1-U) X^{(2)}+g(U),
\end{align}
where $\stackrel{d}{=}$ denotes that both sides have the same distribution; $X^{(1)}, X^{(2)}$ and $U$ are independent; $X^{(0)}$ and $X^{(1)}$ have the same distribution as $X$; $U$ is uniformly distributed on $(0,1)$; and $g(u):=2u\log u + 2(1-u)\log(1-u)+1$. 

Based on this distributional equation, Fill and Janson show that $X$ admits a Schwartz density \citep[Theorem 3.1]{fiJa00}. 
With \cref{thm_A_conds} we can extend this result to a bivariate limit:

The joint distribution of key comparisons and key exchanges performed by Quicksort\footnote{when the input is assumed to be a random permutation} also converges to a limit $(X,Y)$ that can be described by the following stochastic fixed point equation \citep{Ne01}:
\begin{align}\label{rec_QS_bivariate}
\begin{pmatrix} X \\ Y \end{pmatrix}\stackrel{d}{=} \begin{pmatrix} U & 0\\ 0 & U \end{pmatrix} \begin{pmatrix} X^{(1)} \\ Y^{(1)} \end{pmatrix} + \begin{pmatrix} 1-U & 0 \\ 0 & 1-U\end{pmatrix} \begin{pmatrix} X^{(2)} \\ Y^{(2)} \end{pmatrix} + g_2(U)
\end{align}
where $(X^{(1)},Y^{(1)}), (X^{(2)},Y^{(2)})$ and $U$ are independent; $(X^{(j)},Y^{(j)})$ has the same distribution as $(X,Y)$; $U$ is uniformly distributed on $(0,1)$; and 
\begin{align*}
g_2(U):=\left(2 U  \log U +2(1-U)\log(1-U)\right)\begin{pmatrix} 2 \\ 1/3\end{pmatrix}+\begin{pmatrix} 1 \\ U(1-U) \end{pmatrix}.
\end{align*}

To the best of this author's knowledge, the resulting limit $(X,Y)$ has not been studied so far. \cref{thm_A_conds} yields the following:
\begin{thm}\label{thm_Quicksort_bivariate}
	Let $(X,Y)$ be a solution to \eqref{rec_QS_bivariate}. Then $(X,Y)$ admits a bounded density $f\in\mathcal{C}^\infty(\R^2)$.
\end{thm}
\begin{rem}
	\cref{thm_A_conds} also yields that the density $f$ above is a Schwartz function if $\E[|X|^p]<\infty$ and $\E[|Y|^p]<\infty$ for all $p\geq 1$. $X$ is known to have a finite moment generating function, but a corresponding result for $Y$ has not be proven yet. We leave it as an open problem whether all moments of $Y$ are finite.  
\end{rem}
\begin{proof}[Proof of \cref{thm_Quicksort_bivariate}] The assertion follows from \cref{thm_A_conds} if \eqref{s_cond_2}-\eqref{s_cond_6} hold.
	First note that in the case of the bivariate Quicksort limit
	\begin{align*}
	\amax_1=\max\{U,1-U\},\qquad \asec_1=\min\{U,1-U\}.
	\end{align*}
	Conditions \eqref{s_cond_2},\eqref{s_cond_3},\eqref{s_cond_4} and \eqref{s_cond_6} can be verified easily and are left to the reader. For \eqref{s_cond_5} let $(x,y)$ be an arbitrary element of $\Supp((X,Y))$. Then, by \eqref{rec_QS_bivariate},
	\begin{align*}
	\left\{\begin{pmatrix} x \\ y \end{pmatrix}+ g_2(u) : u\in [0,1]\right\}\subset \Supp\left(\begin{pmatrix} X \\ Y \end{pmatrix}\right).
	\end{align*}
	The points $(x,y)^T$, $(x,y)^T+g_2(0)$ and $(x,y)^T +g_2(1/2)$ are in general position. Therefore \eqref{s_cond_5} holds and \cref{thm_A_conds} yields the assertion.
\end{proof}

\noindent
{\bf 2. Large P{\'o}lya Urns.} Consider an urn process with balls of $q$ different colors labeled $1,\ldots,q$. The process evolves in discrete time. 
Let $X_{n,j}$ denote the number of balls of color $j$ in the urn at time step $n$.
Given an initial composition $(X_{0,1},\ldots,X_{0,q})$ with at least one ball and a replacement matrix $(\xi_{i,j})_{i,j=1,\ldots,q}$ (each of them can be deterministic or random) the urn evolves in time as follows. Let $X_n=(X_{n,1},\ldots,X_{n,q})$ be the current composition of balls in the urn. Draw a ball form the urn uniformly at random and denote its color by $I$. Then $X_{n+1}:= X_n+(\xi_{I,1}^\prime,\ldots,\xi_{I,q}^\prime)$ where $(\xi_{i,j}^\prime)_{i,j=1,\ldots,q}$ is an independent copy of the replacement matrix $(\xi_{i,j})_{i,j=1,\ldots,q}$ (and also independent from $X_n$). A replacement rule is called random, if $(\xi_{i,j})_{i,j=1,\ldots,q}$ has a non-degenerate distribution, and deterministic otherwise.

The literature on P{\'o}lya urns is vast. We mainly focus on the results
derived by the contraction method \citep{KnNe14,ChMaPou13,Ma15} since such limit laws are given by stochastic fixed point equations. For more information on P{\'o}lya urns see, e.g., the monographs of Johnson and Kotz \citep{JoKo77},
Mahmoud \citep{Ma09},  the papers of
Janson \citep{Ja04}, Flajolet, Gabarr{\'o} and Pekari \citep{FlGaPe05}, and Pouyanne \citep{Pou08}, as well as the references therein.\\

\noindent
{\bf 2.1.~Large Urns with Two Colors and Deterministic Replacement.} Consider a P{\'o}lya urn with $2$ colors and replacement rule (cf.~\citep[Section 6.1]{KnNe14})
\begin{align*}
(\xi_{i,j})_{i,j\in\{1,2\}}=\begin{pmatrix} a & b \\ c & d \end{pmatrix}
\end{align*}
with constants $a,b,c,d\in\N_0$. Assume $a+b=c+d$ and $bc>0$. Let $K:=a+b+1$ and $\lambda:=(a-c)/(a+b)$. Let $(X_{n}^{[1]})_{n\geq 0}$ denote the urn process with that replacement rule and initial configuration $X_{0}^{[1]}=(1,0)$. Similarly, let $(X_{n}^{[2]})_{n\geq 0}$ be the process with initial configuration $X_{0}^{[2]}=(0,1)$. Then, if $\lambda>1/2$, almost surely and in $L_p$ for all $p\geq 1$ \cite{Ja04,Pou08},
\begin{align}\label{lim_det_repl}
\frac{X_{n,1}^{[1]}-\E[X_{n,1}^{[1]}]}{n^\lambda} \longrightarrow X_1,\qquad
\frac{X_{n,1}^{[2]}-\E[X_{n,1}^{[2]}]}{n^\lambda} \longrightarrow X_2.
\end{align}
The distributions of $X_1$ and $X_2$ are the unique pair of distributions with $\E[X_1]=\E[X_2]=0$, finite variance, and \citep[Theorem 6.1]{KnNe14} 
\begin{equation}\label{rec_det_repl}
\begin{aligned}
X_1\stackrel{d}{=}\sum_{j=1}^{a+1} D_j^\lambda X_1^{(j)} + \sum_{j=a+2}^{K} D_j^\lambda X_2^{(j)} +b_1({\bf D}),\\
X_2\stackrel{d}{=}\sum_{j=1}^c D_j^\lambda X_1^{(j)} + \sum_{j=c+1}^K D_j^\lambda X_2^{(j)}+b_2({\bf D})
\end{aligned}
\end{equation}
where 
\begin{itemize}
	\item  $X_1^{(1)},\ldots X_1^{(K)},X_2^{(1)},\ldots,X_2^{(K)}$ and ${\bf D}:=(D_1,\ldots,D_K)$ are independent; 
	\item $X_i^{(j)}\stackrel{d}{=} X_i$ for all $i,j$;
	\item ${\bf D}$ has the $\mathrm{Dirichlet}\left(\frac 1 {K-1},\ldots, \frac 1 {K-1}\right)$ distribution;
	\item $b_1$ and $b_2$ are some deterministic function with $\Prob(b_j({\bf D})\neq 0)>0$ (the explicit functions can be found in  \citep[Theorem 6.1]{KnNe14}).
\end{itemize} 
\begin{rem}
	Note that the distributions of $X_1$ and $X_2$ are non-degenerate (and thus condition \eqref{s_cond_5} is fulfilled), since the only degenerate, centred solution would be the constant $0$, which leads to a contradiction to $\Prob(b_j({\bf D})\neq 0)>0$.
\end{rem}
Chauvin, Mailler and Pouyanne \cite{ChMaPou13} analyze the limit distributions $X_1$ and $X_2$ in \eqref{lim_det_repl} regarding absolute continuity and finiteness of moments. In particular, they show that $X_1$ and $X_2$ admit bounded, continuous densities. \cref{thm_A_conds} yields the following:
\begin{thm}\label{thm_det_repl}
	The limits $X_1$ and $X_2$ in \eqref{lim_det_repl} admit Schwartz densities.
\end{thm}
\begin{proof}
	In this example $m=2$, $\amax_r$ and $\asec_r$ are the two largest elements of $(D_j^\lambda)_{j\in [K]}$, where $(D_1,\ldots,D_K)$ has a Dirichlet distribution.
	Checking \eqref{s_cond_2}-\eqref{s_cond_6} is left to the reader. The additional integrability condition to obtain a Schwartz density holds by \cite{ChMaPou13}.
\end{proof}

\noindent
{\bf 2.2.~A Random Replacement Urn.} Consider a P{\'o}lya urn with $2$ colors and replacement rule (cf.~\citep[Section 6.2]{KnNe14})
\begin{align*}
(\xi_{i,j})_{i,j\in\{1,2\}}=\begin{pmatrix} F_{p_1} & 1-F_{p_1} \\ 1-F_{p_2} & F_{p_2} \end{pmatrix}
\end{align*}
where $\Prob(F_x = 1)=x =1-\Prob(F_x=0)$ for $x\in\{p_1,p_2\}$. Let $(X_{n}^{[1]})_{n\geq 0}$ denote the urn process with that replacement rule and initial configuration $X_{0}^{[1]}=(1,0)$. Similarly, let $(X_{n}^{[2]})_{n\geq 0}$ be the process with initial configuration $X_{0}^{[2]}=(0,1)$. Finally, let $\lambda=p_1+p_2-1$. Then, if $1/2<\lambda<1$, almost surely and in $L_p$ for all $p\geq 1$ \cite{Ja04}, 
\begin{align}\label{lims_rand_repl_urn}
\frac{X_{n,1}^{[1]}-\E\left[X_{n,1}^{[1]}\right]}{n^\lambda}\longrightarrow X_1,\qquad \frac{X_{n,1}^{[2]}-\E\left[X_{n,1}^{[2]}\right]}{n^\lambda}\longrightarrow X_2
\end{align}
with limiting distributions that can be characterized \citep[Theorem 6.4]{KnNe14} as the unique pair of distributions having finite second moments, $\E[X_1]=\E[X_2]=0$, and satisfying 
\begin{equation}\label{DE_random_repl_urn}
\begin{aligned}
X_1\stackrel{d}{=} U^\lambda X_1^{(1)}+ F_{p_1} (1-U)^\lambda X_1^{(2)}+ (1-F_{p_1})(1-U)^\lambda X_2 +b_1^\prime(U,F_{p_1}),\\
X_2\stackrel{d}{=} U^\lambda X_2^{(1)}+ F_{p_2} (1-U)^\lambda X_2^{(2)}+ (1-F_{p_2})(1-U)^\lambda X_1 +b_2^\prime(U,F_{p_2}),
\end{aligned}
\end{equation}
where
\begin{itemize}
	\item $X_1,X_2, X_1^{(1)},X_1^{(2)},X_2^{(1)},X_2^{(2)}, F_{p_1}, F_{p_2}$ and $U$ are independent;
	\item $X_i^{(j)}\stackrel{d}{=} X_i$ for all $i,j$;
	\item $U$ is uniformly distributed on $[0,1]$; 
	\item $b_1^\prime$ and $b_2^\prime$ are some deterministic functions with $\Prob(b_j^\prime(U,F_{p_j})\neq 0)>0$ (the explicit functions can be found in \citep[Theorem 6.4]{KnNe14})
\end{itemize}
To the best of this author's knowledge the distributions of $X_1$ and $X_2$ have not been studied in the literature. \cref{thm_A_conds} yields the following:
\begin{thm}\label{thm_random_replacement}
	The limits $X_1$ and $X_2$ in \eqref{lims_rand_repl_urn} admit bounded densities  $f_1,f_2\in\mathcal{C}^\infty(\R)$.
\end{thm}
\begin{proof}
	As in the previous example, this urn is covered by \cref{thm_A_conds} with parameters $m=2$, $\amax_r=\max\{U^\lambda,(1-U)^\lambda\}$, and $\asec_r=\min\{U^\lambda,(1-U)^\lambda\}$. Checking \eqref{s_cond_2}-\eqref{s_cond_6} is left to the reader.
\end{proof}
\begin{rem}
Note that $f_1$ and $f_2$ in \cref{thm_random_replacement} are Schwartz functions if $\E[|X_1|^p]$ and $\E[|X_2|^p]$ are both finite for every $p>0$ (see \cref{thm_A_conds}). Checking finiteness of moments will not be done in this paper, although it most likely holds, e.g., by a generalization of the methods of R{\"o}sler  \citep[Theorem 6]{Roe92}.
\end{rem}
\noindent{\bf 2.3.~Large Urns with more than Two Colors.} We end the examples of P{\'o}lya urns with a brief discussion on the case $q\geq 3$. As in Example 2.1, let $X_n^{[j]}$ be the urn composition after $n$ steps when starting with a single ball of color $j$. Limit theorems for $X_n^{[j]}$ are often described by considering the projections of $X_n^{[j]}$ to eigenspaces of the replacement matrix $(\xi_{i,j})_{i,j=1,\ldots,q}$; see, e.g., \cite{Pou08,Ma15}. Now assume the replacement matrix is deterministic and the urn is balanced, i.e.~there is an integer $S$ such that $\sum_j \xi_{i,j}=S$ for all colors $i$. Moreover, let $\lambda$ be a large eigenvalue of the replacement matrix, i.e.~an eigenvalue $\lambda\neq S$ with $\Re(\lambda)>S/2$. Then, properly rescaled, the projection of $X_n^{[j]}$ to the eigenspace of $\lambda$ converges to a limit $X_j$; cf.~\cite{Pou08} or \cite[Theorem 3]{Ma15} for details. Under suitable assumptions \cite[Theorem 8]{Ma15} the limits $X_1,\ldots,X_q$ satisfy
\begin{align}\label{rec_multi_polya}
X_r\stackrel{d}{=}\sum_{j+1}^{S+1} \left(D_j^{(r)}\right)^{\lambda/S} X_{\ell_r(j)}^{(j)},\quad r\in[q],
\end{align}
with the usual independence assumptions and where $D^{(r)}$ is a Dirichlet distributed random vector (the explicit parameters are given in \cite{Ma15}). If $\Im(\lambda)=0$ these equations can be treated in the same fashion as in Example 2.1 and we obtain:
\begin{thm}
	Let $\lambda$ be a large eigenvalue of the replacement matrix with $\Im(\lambda)=0$. Then the limits $X_1,\ldots,X_q$ given by \eqref{rec_multi_polya} are real valued and admit Schwartz densities.
\end{thm}
\begin{proof}
	It is shown by Mailler \cite[Theorem 11]{Ma15} that the support of these limits is $\R$ if $\Im(\lambda)=0$ (it is also shown that they admit densities). In particular \eqref{s_cond_5} holds. The other conditions can be verified as in Example 2.1. Since all moments of the limits are finite \cite{Pou08}, \cref{thm_A_conds} yields the existence of Schwartz densities.
\end{proof}
The case $\Im (\lambda)\neq 0$ needs to be treated slightly differently since the limits in this case are $\C$-valued. However, we also obtain:
\begin{thm}
	Let $\lambda$ be a large eigenvalue of the replacement matrix with $\Im(\lambda)\neq 0$. Then the limits $X_1,\ldots,X_q$ given by \eqref{rec_multi_polya} are $\C$-valued and admit Schwartz densities\footnote{A $\C$-valued random variable $X$ admits a Schwartz density if the vector $(\Re(X),\Im(X))$ admits a Schwartz density.}.
\end{thm}
\begin{proof}
	Recall that $X_1,\ldots,X_q$ can be embedded into $\R^2$ in the canonical way (cf.~\cref{embed_complex}).
	It is not hard to show that $\Supp(X_r)=\C$ for every $r\in [q]$ if $\Im(\lambda)\neq 0$; cf.~\cite[Theorem 11]{Ma15}. In particular, \eqref{s_cond_5} holds. The other conditions can be verified as before, noting that $|(D_j^{(r)})^\lambda|=(D_j^{(r)})^{\Re(\lambda)}$. Since all moments of the limits are finite \cite{Pou08}, \cref{thm_A_conds} yields the existence of Schwartz densities.
\end{proof}
\begin{rem}
	Note that equations like \eqref{rec_multi_polya} also appear in the context of $m$-ary search trees; see \cite{FiKa05,ChLiPou12} for some examples. \cref{thm_A_conds} can be applied to these equations as well but details are left to the reader. 
\end{rem}
	
\noindent
{\bf 3. Path Length in Random Trees.} Let $\mathcal{T}$ be a rooted tree and let $V(\mathcal{T})$ denote the vertex-set of $\mathcal{T}$. Moreover, let $d_v$ denote the distance between $v\in V(\mathcal{T})$ and the root of $\mathcal{T}$, where the distance of two vertices is defined as the number of edges in the unique path connecting them. 

If $\mathcal{T}$ is a tree storing data in its vertices, the total path length can be defined in two different ways: either with respect to the data, or with respect to the vertices.
The total path length of $\mathcal{T}$ with respect to its vertices is
\begin{align*}
\Upsilon(\mathcal{T}):=\sum_{v\in V(\mathcal{T})} d_v.
\end{align*}
Now assume $\mathcal{T}$ stores data in its vertices, e.g.~assume that $\mathcal{T}$ stores $n(\mathcal{T})$ numbers $u_1,\ldots,u_{n(\mathcal{T})}$. Let $v_j$ be the the vertex that contains $u_j$ and let $d_j:=d_{v_j}$. Then the total path length of $\mathcal{T}$ with respect to its data is 
\begin{align*}
\Psi(\mathcal{T}):=\sum_{j=1}^{n(\mathcal{T})} d_j.
\end{align*}
The contraction method has been applied to the total path length of a variety of random trees. We only list some random trees that are covered by our main result (\cref{thm_A_conds}). Definitions and limit laws can be found in the references:\\

\noindent
{\bf 3.1. Random Recursive Trees.} Let $\mathcal{T}_n$ be a random recursive tree with $n$ vertices (see Smythe and Mahmoud \citep{SmyMah95} for a survey on recursive trees).
Then \citep{Ma91,DoFi99}, as $n\rightarrow\infty$, 
\begin{align}\label{LimRRT}
\frac{\Upsilon(\mathcal{T}_n)-\E[\Upsilon(\mathcal{T}_n)]} n \longrightarrow X \quad \text{a.s. and in $L_p$ for any $p>0$}
\end{align}
where $X$ is some non-degenerate random variable. The limit $X$ satisfies \citep{DoFi99}:
\begin{align}\label{DE_RRT}
X\stackrel{d}{=} U X^{(1)} +(1-U)X^{(2)}+h(U),
\end{align}
where $X^{(1)},X^{(2)}$ and $U$ are independent; $X^{(1)}$ and $X^{(2)}$ have the same distribution as $X$; $U$ is uniformly distributed on $[0,1]$; and $h(u):=u+u\log u + (1-u)\log(1-u)$.

This equation is very similar to Equation \eqref{DE_Quicksort} for Quicksort.
As already mentioned by Dobrow and Fill \citep{DoFi99}, essentially the same analysis as for Quicksort show that $X$ admits a Schwartz density. \cref{thm_A_conds} adds a formal proof to this observation:
\begin{thm}\label{thm_RRT} Let $X$ be the limit in \eqref{LimRRT}. Then $X$ admits a Schwartz density.
\end{thm}
\begin{proof}
	The arguments are essentially the same as in the Quicksort example (note that $\amax_1$ and $\asec_1$ coincide with the coefficients in Quicksort). Details are left to the reader. For the finiteness of all moments see Dobrow and Fill \citep{DoFi99}.
\end{proof}
A slightly modified tree model with a weighted root (called a {\it Hoppe tree}) was defined and studied by Leckey and Neininger \citep{LeNe13}. We do not go into detail for this variation and just point out that the limit of the total path length in Hoppe trees is also covered by \cref{thm_A_conds}.\\

\noindent
{\bf 3.2. Split Trees.}  Random split trees are a class of random trees introduced by Devroye \citep{Dev98}. The distribution of a random split tree is determined 
by a {\it branch factor} $b\in \N_{\geq 2}$, a {\it capacity} $s\in\N_0$, a {\it (random) split vector} $\mathcal{V}=(V_1,\ldots,V_b)$, and {\it ball distribution parameters} $(s_0,s_1)\in\N_0^2$. Here, $\mathcal{V}$ is a random variable taking values in the unit simplex of $\R^b$. We refer to Devroye \citep{Dev98} for a definition of a random split tree. 
Let $\mathcal{T}_n$ be a random split tree storing $n$ items. Let $X_n:=\Psi(\mathcal{T}_n)$.

Let $\mu:=-\E[V_1\ln V_1 +\cdots + V_b \ln V_b]$ and assume $\mu\neq 0$ (i.e.~$\Prob( \exists i : V_i=1)<1$). Let 
\begin{align}\label{def_split_CV}
C(\mathcal{V}):=1+\frac 1 \mu \sum_{j=1}^b V_j\ln V_j.
\end{align}
Based on an unproven assumption on the asymptotic expansion of the mean, Neininger and R{\"u}schendorf \citep{NeRue99} show
\begin{align}\label{lim_split_tree}
\frac{ X_n -\E[X_n]} n \stackrel{d}{\longrightarrow} X.
\end{align}
This assumption on the mean has been verified by Broutin and Holmgren \citep{BrouHolm} for split vectors with $\Prob(C(\mathcal{V})\neq 0)>0$ (see also Munsonius \citep{Mun11} for split vectors where the marginals have a Lebesgue density). The limit $X$ then satisfies
\begin{align}\label{DE_splittree}
X\stackrel{d}{=}\sum_{j=1}^b V_j X^{(j)} +C(\mathcal{V}),
\end{align}
where $X^{(1)},\ldots, X^{(b)}$ and $\mathcal{V}$ are independent; and $X^{(j)}$ has the same distribution as $X$.

To the best of this author's knowledge the distribution of $X$ in \eqref{lim_split_tree} has not been studied in the literature. The main result in the next section yields the following:
\begin{thm}\label{thm_split}
	Assume that the split vector $\mathcal{V}=(V_1,\ldots,V_b)$ satisfies
	\begin{itemize}
		\item $\Prob(\max_j V_j \geq 1-x)\leq \lambda x^\nu\quad \text{for some $\lambda,\nu>0$ and all $x\geq 0$,}$
		\item $\Prob(C(\mathcal{V})\neq 0)>0.$
	\end{itemize}
	Then the limit $X$ in \eqref{lim_split_tree} admits a Schwartz density $f$. 
\end{thm}
\begin{proof}
	Note that 
	since $\sum_j V_j =1$, 
	\begin{align*}
	\amax_1\geq 1 / b, \qquad \asec_1\geq (1-\amax_1)/(b-1).
	\end{align*}
	Hence \eqref{s_cond_2} and \eqref{s_cond_3} hold by assumption on $\max_j V_j$. Moreover, \eqref{s_cond_4} holds since $(V_1,\ldots,V_b)$ is a probability vector and \eqref{s_cond_6} holds since $\Prob(\max_j V_j =1)=0$. Finally note that \eqref{DE_splittree} and $\Prob(C(\mathcal{V})\neq 0)>0$ imply that $X$ is non-degenerate and therefore \eqref{s_cond_5} holds.
	\cref{thm_A_conds} yields that $X$ admits a Schwartz density, since all moments of $X$ are finite \citep[Theorem 2.1]{BrouHolm}.	
\end{proof}
Note that the second condition ($\Prob(C(\mathcal{V})\neq 0)>0$) is only required for the limit law by Broutin and Holmgrem. More generally, the first condition in the previous theorem implies that any non-degenerate solution to \eqref{DE_splittree} admits a bounded density $f\in\mathcal{C}^\infty(\R)$. If the solution has finite moments of any order, then $f$ is Schwartz function.

Also note that the class of split trees covers several random trees appearing in context of computer science. A list of examples is given by Devroye \citep[Table 1]{Dev98}. All examples satisfy the (first) condition of \cref{thm_split} and thus all limits in these examples admit Schwartz densities.

Finally note that \cref{thm_A_conds} can also be applied to multivariate limit laws. As an example we discuss a bivariate limit law for split trees by Munsonius \citep[Theorem 1.5]{Mun11}.
Again let $\mathcal{T}_n$ be a random split tree storing $n$ items and let $X_n:=\Psi(\mathcal{T}_n)$. Moreover, let $W_n$ denote the Wiener index of $\mathcal{T}_n$ (see \citep{Mun11}). Assume that the marginals $V_j$ of the split vector $\mathcal{V}$ have Lebesgue densities. Then \citep[Theorem 1.5]{Mun11}
\begin{align}\label{lim_WX}
\left(\frac{ W_n -\E[W_n]} {n^2} , \frac{ X_n-\E[X_n]} n \right)\stackrel{d}{\longrightarrow} (W,X)
\end{align}
for some limit $(W,X)$. This limit satisfies $\E[W]=\E[X]=0$ and 
\begin{align}\label{rec_WX}
\begin{pmatrix} W \\ X \end{pmatrix} \stackrel{d}{=}\sum_{j=1}^b\begin{pmatrix} V_j^2 & V_j(1-V_j) \\ 0 & V_j \end{pmatrix} \begin{pmatrix} W^{(j)} \\ X^{(j)}\end{pmatrix} + \eta(\mathcal{V})
\end{align}
where $\mathcal{V}, (W^{(1)},X^{(1)}),\ldots,(W^{(b)},X^{(b)})$ are independent; $(W^{(j)},X^{(j)})\stackrel{d}{=}(W,X)$; and 
\begin{align*}
\eta(\mathcal{V})=C(\mathcal{V}) \begin{pmatrix} 1 \\ 1 \end{pmatrix} + \begin{pmatrix} c \left(1 - \sum_j V_j^2\right)-1 \\ 0 \end{pmatrix}
\end{align*}
for some constant $c\in\R$ and with $C(\mathcal{V})$ as in \eqref{def_split_CV}. \cref{thm_A_conds} yields the following:
\begin{thm}\label{thm_WX} Let $(V_1,\ldots,V_b)$ be a split vector with
	\begin{itemize}
		\item $\Prob\left(\max_j V_j \geq 1-x\right)\leq \lambda x^\nu$ for some $\lambda,\nu>0$ and all $x\geq 0$,\\
		\item $\Prob\left(\sum_j V_j^2=x\right)<1$ for all $x\in \R$.
	\end{itemize}
	Then the limit $(W,X)$ in \eqref{lim_WX} admits a bounded density $f\in\mathcal{C}^\infty(\R^2)$.
\end{thm}
\begin{rem}
	\cref{thm_A_conds} also yields that the density $f$ above is a Schwartz function, if $\E[|X|^p]<\infty$ and $\E[|W|^p]<\infty$ for all $p>0$. It is known that $X$ has finite exponential moments \citep[Theorem 5.1]{NeRue99}\citep[Theorem 2.1]{BrouHolm}. We leave it as an open problem whether all moments of $W$ are finite. 
\end{rem}
\begin{rem}\label{rem_sec_cond_split}
	The second condition in \cref{thm_WX} is necessary, since otherwise $(X,X)$ is a solution to \eqref{rec_WX}: Note that $\E[W]=\E[X]=\E[C(\mathcal{V})]=0$ and thus
	\begin{align*}
	\E\left[c\left(1-\sum_j V_j^2\right)\right] = 1.
	\end{align*}
	Hence, if $\sum_j V_j^2$ has a degenerate distribution then $\eta(\mathcal{V})= C(\mathcal{V}) (1,1)^T$. Therefore if $X$ is the solution to \eqref{DE_splittree} then $(X,X)$ is a solution to \eqref{rec_WX}.
\end{rem}
\begin{proof}[Proof of \cref{thm_WX}] Recall that in this example $m=1$ and 
	\begin{align*}
	A_{1,j}&=\begin{pmatrix} V_j^2 & V_j (1-V_j) \\ 0 & V_j \end{pmatrix},\, j\in [b],\quad A_{1,j}={\bf 0},\, j>b,\\ 
	b_1&=\eta(\mathcal{V})=C(\mathcal{V})\begin{pmatrix} 1 \\ 1\end{pmatrix} +\begin{pmatrix} c \left(1-\sum_j V_j^2\right) -1 \\ 0 \end{pmatrix}.
	\end{align*}
	As a preparation we need to compute $\|A_{r,j}^T\|_{\mathrm{op}}$ and $\alpha_{r,j}:=\min_{\|t\|=1} \|A_{r,j}^T t \|$. Let 
	\begin{align*}
	f_v(x):=\left\| \begin{pmatrix} v^2 & 0 \\ v(1-v) & v \end{pmatrix} \begin{pmatrix} x \\ \sqrt{1-x^2}\end{pmatrix} \right\|^2,\qquad v\in[0,1], x\in [-1,1].
	\end{align*}
	Rewriting $f_v(x)=v^2\left[ 1 - 2(1-v)\left(v x^2 - x \sqrt{1-x^2}\right)\right]$ reveals that it is sufficient to maximize and minimize the function
	\begin{align*}
	x\mapsto v x^2 -x\sqrt{1-x^2},\qquad x\in [-1,1].
	\end{align*}
	The extreme points of this function in $(-1,1)$ are
	\begin{align*}
	x_1=\frac 1 {\sqrt{ 2} } \sqrt{ 1 - \frac{ v } {\sqrt{1+v^2}}},\qquad x_2=-\frac 1 {\sqrt{2}}\sqrt{1+ \frac v {\sqrt{1+v^2}}}.
	\end{align*}
	Hence,
	\begin{align*}
	\max_{x\in[-1,1]} f_v(x)&=v\sqrt{1-v(1-v)+(1-v)\sqrt{ 1 +v^2}},\\
	\min_{x\in[-1,1]} f_v(x)&=v\sqrt{1-v(1-v)-(1-v)\sqrt{ 1 +v^2}}
	\end{align*}
	and therefore, for any $j\in [b]$,
	\begin{align*}
	\|A_{1,j}^T\|_{\mathrm{op}} &= V_j\sqrt{1-V_j(1-V_j)+(1-V_j)\sqrt{ 1 +V_j^2}}\\
	\alpha_{1,j}&=V_j\sqrt{1-V_j(1-V_j)-(1-V_j)\sqrt{ 1 +V_j^2}}
	\end{align*}
	In particular, using $(a)$ $\sqrt{1+v^2}\leq 1+v$ and $v(2-v)\leq 1$, and $(b)$ $\sqrt{1+v^2}\leq 1+v^2$,
	\begin{align*}
	\|A_{1,j}^T\|_{\mathrm{op}}\stackrel{(a)}{\leq} \sqrt{V_j}\leq 1,\qquad \alpha_{1,j}\stackrel{(b)}{\geq } V_j^{5/2}.
	\end{align*}
	With these bounds combined with the fact that $(V_1,\ldots,V_b)$ is a probability vector it is not hard to check  \eqref{s_cond_2}-\eqref{s_cond_4} and \eqref{s_cond_6} (cf.~the previous proof for similar arguments).
	
	It remains to show that $\Supp((W,X))$ is in general position.
	Note that in $\R^2$ only (subsets of) lines are not in general position. Thus $\Supp((W,X))$ is in general position if and only if
	\begin{align*}
	\Prob(X=0)<1\quad \text{and}\quad\Prob(W=aX+b)<1\text{ for all }a,b\in\R.
	\end{align*}
	We already know $\Prob(X=0)=0$ by the previous result on the path length. Also note that $\E[W]=\E[X]=0$ implies $\Prob(W=aX+b)<1$ for all $b\neq 0$. 
	Thus, it only remains to show $\Prob(W=aX)<1$ for all $a\in\R$.
	
	Suppose for the sake of contradiction that there is an $a\in\R$ such that
	\begin{align*}
	\Prob(W=aX)=1.
	\end{align*}
	Then $\Supp((W,X))=\{ x (a,1) : x\in\Supp(X)\}$. However,  \eqref{rec_WX} implies for any $(ax,x)\in\Supp((W,X))$ and any $v\in\Supp(\mathcal{V})$ that
	\begin{align*}
	\Supp((W,X))\ni \zeta(x,v)&:=\sum_{j=1}^b\begin{pmatrix} v_j^2 & v_j(1-v_j) \\ 0 & v_j \end{pmatrix} \begin{pmatrix} ax \\ x\end{pmatrix} + \eta(v).
	\end{align*}
	In particular, $\zeta(x^\prime,v)-\zeta(x,v)\in\{ y (a,1): y\in\R\}$ for any $x,x^\prime\in\Supp(X)$ and $v\in\Supp(\mathcal{V})$, since $\{ y (a,1): y\in\R\}$ is closed under subtraction. Hence
	\begin{align*}
	(x-x^\prime)\sum_{j=1}^b\begin{pmatrix} v_j^2 & v_j(1-v_j) \\ 0 & v_j \end{pmatrix} \begin{pmatrix} a \\ 1\end{pmatrix}\in \{ y (a,1): y\in\R\},
	\end{align*}
	which for $x\neq x^\prime$ is equivalent to
	\begin{align*}
	\sum_j v_j^2 a + 1-\sum_j v_j^2 = a.
	\end{align*}
	Since $\Prob(\sum_j V_j^2<1)=\Prob(\max_j V_j <1)=1$ by assumption, this implies $a=1$. Thus we may conclude $\Prob(W=aX)<1$ for all $a\neq 1$ and it only remains to show $\Prob(W=X)<1$.
	
	As before, note that if $(x,x)\in\Supp((W,X))$ then
	\begin{align*}
	\Supp((W,X))\ni\zeta(x,v)&:=\sum_{j=1}^b\begin{pmatrix} v_j^2 & v_j(1-v_j) \\ 0 & v_j \end{pmatrix} \begin{pmatrix} x \\ x\end{pmatrix} + \eta(v)=\begin{pmatrix} x \\ x\end{pmatrix} + \eta(v),
	\end{align*}
	which is only possible if $\eta(v)\in\{y(1,1) : y\in\R\}$. However, by assumption there is a $(v_1,\ldots,v_d)\in\Supp(\mathcal{V})$ with
	\begin{align*}
	c\left(1-\sum_j v_j^2\right)-1\neq 0
	\end{align*}
	and thus $\eta(v)\notin\{y(1,1) : y\in\R\}$, a contradiction to $\Supp((W,X))\subset\{y(1,1) : y\in\R\}$. Thus we may conclude that $\Supp((W,X))$ is in general position. \cref{thm_A_conds} yields the assertion. 
\end{proof}

\section{Proofs}\label{sec_proofs}
\noindent
This section contains the proofs of the results in \cref{sec_main_thm}.
Recall $\phi_r(t):=\E[\exp(i\langle t, X_r\rangle)]$ for $t\in\R^d$.
First note that if $\phi_r$ is integrable with respect to the Lebesgue measure $\lambda^d$ on $\R^d$, then the distribution of $X_r$ admits a bounded density function $f_r$ given by the Fourier inversion formula:
\begin{align}\label{fourier_inv}
f_r(t)=
\frac 1 {(2\pi)^d} \int_{\R^d} \mathrm{e}^{-i\langle x,t\rangle} \phi_r(x) \mathrm{d}\lambda^d(x). 
\end{align}
Moreover, for any $\beta=(\beta_1,\ldots,\beta_d)\in\N_0^d$, a standard argument based on the Dominated Convergence Theorem reveals that the derivative $D^\beta f_r$ exists and is continuous if $x\mapsto x^\beta \phi_r (x)$ is integrable, where $x^\beta:=x_1^{\beta_1}\cdots x_d^{\beta_d}$ (see the proof of Lemma 2 in \citep[ Section XV.4]{fe71}, for example). Thus, the following implication holds for any $\eta>d$:
\begin{align}\label{tail_phi_impl_density}
|\phi_r(t)|=\bo\left(\|x\|^{-\eta}\right)\quad \Longrightarrow \quad \text{$X_r$ admits a density $f_r\in\mathcal{C}^{\lceil \eta\rceil -d-1}(\R^d)$.}
\end{align}
This observation can be extended to show that the class of Schwartz functions is preserved under Fourier transformation \citep[Theorem 7.4(d)]{Ru91}:
\begin{align}\label{iff_Schwartz_space}
\phi_r\text{ is a Schwartz function}\quad \Longleftrightarrow \quad X_r\text{ admits a Schwartz density.}
\end{align}
The upcoming bounds on $|\phi_r|$ are based on the following observation:
\begin{lem}\label{rec_char_fct}
Assume that $\eqref{main_rec}$ holds. Then 
\begin{align*}
|\phi_r(t)|\leq \E\left[\prod_{j=1}^{\infty} \left|\phi_{\ell_r(j)}\left(A_{r,j}^T t\right)\right|\right]
\quad \text{for all $t\in\R^d$ and $r\in[m]$}.
\end{align*}
Moreover, let $\psi_r(x):=\sup\limits_{\|t\|\geq x} | \phi_r(t)|$. Then,
\begin{align*}
\psi_r(x)\leq \E\left[\prod_{j=1}^{\infty} \psi_{\ell_r(j)}(\alpha_{r,j} x)\right]
\quad \text{for all $x\geq 0$ and $r\in[m]$}.
\end{align*}  
\end{lem}
\begin{proof}
Equation \eqref{main_rec}, Jensen's inequality and the independence in \eqref{main_rec} imply
\begin{align*}
 |\phi_r(t)|\leq\E\left[\prod_{j=1}^{\infty} \left|\E\left[\exp(i\langle A_{r,j}^T t, X_{\ell_r(j)}^{(j)}\rangle)  \big| (A_{r,j})_{j\geq 1},b_r\right]\cdot\exp(i\langle t,b_r\rangle) \right|\right],
\end{align*}
in which the exchange of infinite product and conditional expectation also uses the Dominated Convergence Theorem. Note that the remaining conditional expectations equal $\phi_{\ell_r(j)}(A_{r,j}^T t)$ since  $X_{\ell_r(j)}^{(j)}$ and $((A_{r,j})_{j\geq 1},b_r)$ are independent. 
Therefore, the first bound of the claim follows since $|\exp(i\langle t,b_r\rangle)|=1$.

The bound on $\psi_r$ follows from the first result and $\|A_{r,j}^T t\|\geq \alpha_{r,j}  x$ for all $\|t\|\geq x$.
\end{proof}
The remainder of this section contains the missing proofs of Section \ref{sec_main_thm}. For the reader's convenience, Conditions \eqref{c1}-\eqref{c7} and all results are restated in this section.
\begin{defn}\label{def_c_p} 
Conditions \eqref{c1}-\eqref{c3} hold if for all $r\in[m]$:
\begin{flalign}
&\Prob(\amax_r>0)=1 ,&\label{c1p}\tag{C1}\\
&\E[N_r((0,1])]>1 ,&\label{c2p}\tag{C2}\\
&\Prob(\langle s, X_r \rangle \in \Z+c)<1\text{ for all }s\in\R^d\setminus\{{\bf 0}\}\text{ and }c\in\R .&\label{c3p}\tag{C3}
\end{flalign}
Let $\eta>0$. Conditions \eqref{c4}-\eqref{c6} hold for $\eta$ if for all $r\in[m]$:
\begin{flalign}
&\E[(\asec_r)^{-\eta} |\asec_r>0]<\infty ,&\label{c4p}\tag{C4}\\
&\Prob(\amax_r \leq x)= \bo(x^\eta)\text{ as } x\rightarrow 0,& \label{c5p}\tag{C5}\\
&\E[(\amax_r)^{-\eta}\Ind_{\{\asec_r=0\}}] <1.& \label{c6p}\tag{C6}
\end{flalign}
Finally, let $\chi:(0,\infty)\rightarrow(0,\infty)$ be a function. Condition \eqref{c7} holds for $\chi$ if for all $\beta>0$ a constant $C_\beta>0$ exists such that for all $x>0$ and $r\in[m]$
\begin{flalign}
& \E\left[ \prod_{j\geq 1} \left((\alpha_{r,j} x)^{-\beta} \wedge 1 \right)\right]\leq C_{\beta} x^{-\chi(\beta)}. &\label{c7p}\tag{C7}
\end{flalign}
\end{defn}
\begin{lem}\label{lem_conv_zero_p}
Assume \eqref{c1}, \eqref{c2} and \eqref{c3}. Then, 
\begin{align*}
\lim_{R\rightarrow\infty} \sup_{\|t\|=R} \left| \phi_r(t)\right| =0
\text{ for all } r\in[m].
\end{align*}
\end{lem}
As a preparation for the proof, recall the following simple property of complex valued random variables.
\begin{lem}\label{lem_C_leq1}
Let $Z$ be a $\C$-valued random variable with $|Z|\leq 1$ a.s.~and $\E[Z]=1$. Then $Z=1$ almost surely.
In particular, every $\R^d$-valued random variable $X$ with a non-lattice distribution (cf.~\cref{def_nonlattice}) satisfies $|\E[\exp(i\langle t,X\rangle)]<1$ for all $t\neq {\bf 0}$. 
\end{lem}
\begin{proof}
	For the first part note that the conditions on $Z$ imply $\E[\Re(Z)]=1$ and $|\Re(Z)|\leq 1$. Thus $\Re(Z)=1$ almost surely and therefore $\Im(Z)=0$ since $|Z|\leq 1$.
	
	For the second part assume for the sake of contradiction that 
	$\E[\exp(i\langle t ,X\rangle)]=\exp(i\beta)$ for some $t\neq {\bf 0}$ and $\beta\in [0,2\pi)$. Let $X^\prime=\langle t ,X\rangle -\beta$. Then $\E[\exp(iX^\prime)]=1$ and thus $\exp(iX^\prime)=1$ almost surely, contradicting the non-lattice assumption.	
\end{proof}
\begin{proof}[Proof of \cref{lem_conv_zero_p}]
The proof is based on the ideas of \citep[Lemma 6.2]{Don72} (cf.~also \citep[Lemma 3.1]{liu01}):
Let 
\begin{align*}
 g_r:[0,\infty)\rightarrow [0,1],\,R\mapsto \sup_{\|t\|=R} |\phi_r(t)|,\; r\in[m],\qquad g:=\max_{r\in[m]} g_r.
\end{align*}
Lemma \ref{rec_char_fct} implies for any $r\in[m]$ and $t\in\R^d$
\begin{align}
 |\phi_r(t)| \leq \E\left[\prod_{j=1}^{\infty} \left|\phi_{\ell_r(j)}\left(A_{r,j}^T t\right)\right|\right]\leq \E\left[\prod_{j\in I_r} \left|\phi_{\ell_r(j)}\left(A_{r,j}^T t\right)\right|\right]\label{pf_4_1_rec_bound}
\end{align}
with $I_r=\{j\in \N : \alpha_{r,j}>0\}$. Note that $\|A_{r,j}^T t\|\rightarrow\infty$ for all $j\in I_r$ almost surely as $\|t\|\rightarrow\infty$.
Now let\begin{align*}
 \xi_r:=\limsup_{R\rightarrow\infty} g_r(R),\; r\in[m],\qquad \xi=\max_{r\in[m]} \xi_r = \limsup_{R\rightarrow\infty} g(R).
\end{align*}
The choice of $I_r$ yields that almost surely
\begin{align}\label{pf_5_3_added}
\limsup_{t\rightarrow\infty}\prod_{j\in I_r} \left|\phi_{\ell_r(j)}\left(A_{r,j}^T t\right)\right|\leq \limsup_{t\rightarrow\infty} \prod_{j\in I_r} g_{\ell_r(j)}(\|A_{r,j}^T t\|) \leq \prod_{j\in I_r} \xi_{\ell_r(j)} \leq \xi^{|I_r|}.
\end{align}
Now let $t_R^{(r)}$ be chosen in such a way that $\phi_r(t_R^{(r)})=g_r(R)$ for $R\geq 0$. Then, since $g_r$ is bounded by $1$, the Dominated Convergence Theorem and \eqref{pf_4_1_rec_bound} imply
\begin{align*}
\xi_r\leq \E\left[\limsup_{R\rightarrow\infty}\prod_{j\in I_r} \left|\phi_{\ell_r(j)}\left(A_{r,j}^T t_R^{(r)}\right)\right|\right].
\end{align*}
Hence \eqref{pf_5_3_added} yields
\begin{align*}
\xi_r \leq \E\left[\xi^{|I_r|}\right]\leq \Prob(|I_r|=0)+\xi\Prob(|I_r|=1) +\xi^2 \Prob(|I_r|\geq 2),
\end{align*}
in which the second inequality also uses $\xi\in [0,1]$ and thus $\xi^y\leq \xi^{\min(2,y)}$ for $y\in \N$.
Note that $\Prob(|I_r|=0)=0$ by condition \eqref{c1} and that  $\Prob(|I_r|\geq 2)>0$ by condition \eqref{c2} and the fact that $|I_r|\geq N_r((0,1])$. Thus, the previous bound yields 
\begin{align}\label{pf_53_xi}
\xi\leq \Prob(|I_s|=1) \xi + (1-\Prob(|I_s|=1))\xi^2\quad\text{where } s\in\argmax_{r\in [m]} \xi_r. 
\end{align}
Recalling $\Prob(|I_s|=1)<1$ and $\xi\in[0,1]$, \eqref{pf_53_xi} implies $\xi\leq \xi^2$ and therefore $\xi\in\{0,1\}$.

It remains to show $\xi<1$ which is done by contradiction. Observe that condition \eqref{c3} implies 
\begin{align}\label{pf_f_less_1}
g(R)<1\quad\text{ for all $R>0$}
\end{align}
which can be seen as follows: First note that $|\phi_r(t)|<1$ for all $t\neq {\bf 0}$ by \cref{lem_C_leq1} and \eqref{c3}. Since $\{t\in\R^d:\|t\|=R\}$ is a compact set and $t\mapsto|\phi_r(t)|$ is continuous, one obtains \eqref{pf_f_less_1}.
Now suppose for the sake of contradiction that $\xi=1$. Fix any $R_0>0$.
Choose $R_1^{(n)}$ and $R_2^{(n)}$ for all integers $n$ with $1-1/n \geq g(R_0)$ in such a way that the following holds:
\begin{align*}
R_1^{(n)}\leq R_0\leq R_2^{(n)},\quad g\left(R_1^{(n)}\right)=g\left(R_2^{(n)}\right)=1-\frac 1 n,\quad g(R)\leq 1-\frac 1 n \text{ for }R\in \left[R_1^{(n)},R_2^{(n)}\right].
\end{align*}
This is possible since $g$ is continuous and $g(0)=\xi=1$ by assumption. Note that $(R_1^{(n)})_n$ is a nonnegative, non-increasing sequence. Thus $R_1^{(n)}$ converges to a limit $R_1^*$. However, the continuity of $g$ implies $g(R_1^*)=\lim\limits_{n\rightarrow\infty}g(R_1^{(n)})=1$
and therefore $R_1^*=0$ by \eqref{pf_f_less_1}. Moreover, the sequence $(R_2^{(n)})_n$ 
is bounded from below by $R_0$ and thus it
cannot be convergent since this would contradict \eqref{pf_f_less_1} and the continuity of $g$.
Therefore, $R_2^{(n)}\rightarrow\infty$ since $(R_2^{(n)})_n$ is non-decreasing.
Hence, the sequences satisfy, as $n\rightarrow\infty$,
\begin{align}\label{convergence_sequences_R}
 R_1^{(n)}\rightarrow 0 \quad \text{ and }\quad R_2^{(n)}\rightarrow\infty.
\end{align}
Now let $t_n$ and $r_n$ be chosen in such a way that $\|t_n\|=R_2^{(n)}$ and $g(R_2^{(n)})=|\phi_{r_n}(t_n)|$. Then, 
 by Lemma \ref{rec_char_fct},
 \begin{align}\label{contradic_pf_lem_0}
 1-\frac 1 n =|\phi_{r_n}(t_n)|\leq  \E\left[\prod_{j=1}^{\infty} \left|\phi_{\ell_{r_n}(j)}\left(A_{r_n,j}^T t_n\right)\right|\right]
 \leq  \E\left[\left(1-\frac 1 n \right)^{N_{r_n}^{(n)}}\right]
 \end{align}
 with $N_r^{(n)}=\sum_{j\geq 1} \Ind_{\{\alpha_{r,j} \geq R_1^{(n)}/R_2^{(n)}\}\cap\{ \|A_{r,j}^T\|_{\mathrm{op}}\leq 1 \}}$ for $r\in [m]$. Note that $N_r^{(n)}\rightarrow N_r((0,1])$ almost surely as $n\rightarrow\infty$, since $R_1^{(n)}/R_2^{(n)}\rightarrow 0$.
  However, \eqref{contradic_pf_lem_0} and $\E[N_r((0,1])]>1$ lead to a contraction for large $n$, which can be seen as follows:
 
 First let $c, n_0\in\N$ and $\varepsilon>0$ be chosen in such a way that 
 \begin{align}\label{choice_c_n_xxx}
 \E[N_{r_n}^{(n)}\wedge c]>1+\varepsilon\quad \text{for all }n\geq n_0.
 \end{align}
  This is possible due to the fact that (a) $\E[N_r((0,1])]>1$ and the Monotone Convergence Theorem imply $\E[N_r((0,1])\wedge c]>1$ for sufficiently large $c$, and (b) $\E[N_r^{(n)}\wedge c]\rightarrow\E[N_r((0,1])\wedge c]$ as $n\rightarrow\infty$ by the Dominated Convergence Theorem.
  
  Next note that $x\mapsto (1-1/n)^x$ is decreasing in $x$ and that
  \begin{align}
  \left(1-\frac 1 n \right)^x\leq 1-\frac x n + \frac{2^x} {n^2}\quad \text{for } x\in\N,
  \end{align}
  e.g., using the Binomial Theorem. Hence
  \begin{align*}
  \E\left[\left(1-\frac 1 n \right)^{N_{r_n}^{(n)}}\right]\leq \E\left[\left(1-\frac 1 n \right)^{N_{r_n}^{(n)}\wedge c}\right]
  \leq 1-\frac {\E[N_{r_n}^{(n)}\wedge c]} n+\frac{2^c}{n^2},
  \end{align*}
  which is less than $1-1/n$ for large $n$ by \eqref{choice_c_n_xxx}. Therefore \eqref{contradic_pf_lem_0} leads to a contraction and thus $\xi<1$. Since $\xi\in\{0,1\}$, we obtain $\xi=0$ as claimed. 
 \end{proof}

\begin{prop}\label{thm_poly_decay_p} Assume \eqref{c1}-\eqref{c6} with $\eta>0$. Then,
\begin{align*}
|\phi_r(t)|=\bo\left(\|t\|^{-\eta}\right)\text{ for all $r\in[m]$ as } \|t\|\rightarrow\infty.
\end{align*}
If $\eta>1$, then $X_r$ admits a bounded density function $f_r\in\mathcal{C}^{\lceil \eta \rceil -d-1}(\R^d)$ for all $r\in[m]$.
\end{prop}
\begin{proof} 
Let $\psi_r(x)=\sup_{\|t\|\geq x} |\phi_r(t)|$ and  $\psi=\max_{r\in [m]} \psi_r$.
Recall that Lemma \ref{rec_char_fct} yields for $x\geq 0$ and $r\in[m]$
\begin{align}\label{pf_poly_decay_1}
\psi_r(x)\leq \E\left[\prod_{j=1}^{\infty} \psi_{\ell_r(j)}(\alpha_{r,j} x)\right]\leq \E\left[\psi (\amax_r x) \psi(\asec_r x)\right].
\end{align}  
Lemma \ref{lem_conv_zero_p} implies for any $\varepsilon>0$ the existence of a constant $x_0=x_0(\varepsilon)$ such that $\psi(x)\leq \varepsilon$ for $x\geq x_0$. Thus, \eqref{pf_poly_decay_1} yields for $x>0$
\begin{align*}
\psi_r(x)\leq \Prob(\amax_r \leq x_0 /x )
+ \varepsilon \E\left[ \psi(\asec_r x) \Ind_{\{\asec_r>0\}}\right]
+\E\left[\psi (\amax_r x) \Ind_{\{\asec_r=0\}}\right].
\end{align*}
Let $r(x)=\mathrm{arg}\max\limits_{r\in [m]}\E\left[\psi (\amax_r x) \Ind_{\{\asec_r=0\}}\right]$
and $s(x)=\mathrm{arg}\max\limits_{s\in[m]}\E\left[ \psi(\asec_s x) \Ind_{\{\asec_s>0\}}\right]$.
Then, the previous bound and condition \eqref{c5} imply the existence of a constant $C>0$ such that 
\begin{align}\label{pf_poly_decay_2}
\psi(x)\leq C \left(\frac {x_0} x\right)^\eta+ \E\left[\psi (\amax_{r(x)} x) \Ind_{\{\asec_{r(x)}=0\}}\right]+ \varepsilon \E\left[ \psi(\asec_{s(x)} x) \Ind_{\{\asec_{s(x)}>0\}}\right],\quad x>0.
\end{align}
By assumption \eqref{c4} and \eqref{c6} there are constants $c_1\in(0,1)$ and $c_2>0$ such that
\begin{align}\label{pf_poly_decay_3}
\E\left[(\amax_r)^{-\eta}\Ind_{\{\asec_{r}=0\}}\right]\leq c_1,\qquad \E\left[(\asec_{r})^{-\eta}\Ind_{\{\asec_{r}>0\}}\right]\leq c_2,\quad\text{for all $r\in [m]$.}
\end{align}
Moreover, \eqref{c2} implies the existence of another constant $p<1$ such that $\Prob(\asec_r=0)\leq p$ for all $r\in [m]$. Hence, Equation \eqref{pf_poly_decay_2} and the trivial upper bound $\psi\leq 1$ yield
\begin{align*}
\psi(x)\leq C \left(\frac {x_0} x\right)^\eta+ p +\varepsilon,\quad x>0.
\end{align*}
We end the proof by showing by induction on $n$ that 
\begin{align}\label{IV_pf_5_5}
\psi(x)\leq C \left(\frac{x_0} x \right)^\eta \sum_{j=0}^{n-1} (c_1+c_2\varepsilon)^j + (p+\varepsilon)^n,\quad n\in\N.
\end{align}
Note that this implies the assertion when choosing $\varepsilon< \min(1-p, (1-c_1)/c_2)$ and letting $n\rightarrow\infty$. We already deduced \eqref{IV_pf_5_5} for $n=1$. Now assume the bound holds for some $n$. Then, using \eqref{pf_poly_decay_2} and the induction hypothesis,
\begin{align*}
\psi(x)&\leq C \left(\frac {x_0} x\right)^\eta+ 
C \left(\frac {x_0} x\right)^\eta\E\left[(\amax_{r(x)})^{-\eta}\Ind_{\{\asec_{r(x)}=0\}}\right] \sum_{j=0}^{n-1} (c_1+c_2\varepsilon)^j + (p+\varepsilon)^n\Prob\left(\asec_{r(x)}=0\right)\\
&\qquad +
\varepsilon
C \left(\frac {x_0} x\right)^\eta\E\left[(\asec_{s(x)})^{-\eta}\Ind_{\{\asec_{s(x)}>0\}}\right] \sum_{j=0}^{n-1} (c_1+c_2\varepsilon)^j + \varepsilon(p+\varepsilon)^n\Prob\left(\asec_{s(x)}>0\right).
\end{align*}
Hence \eqref{pf_poly_decay_3}, $\Prob\left(\asec_{s(x)}>0\right)\leq 1$ and $\Prob\left(\asec_{r(x)}=0\right)\leq p$ yield
\begin{align*}
\psi(x)\leq C \left(\frac{x_0} x \right)^\eta \sum_{j=0}^{n} (c_1+c_2\varepsilon)^j + (p+\varepsilon)^{n+1}.
\end{align*}
Therefore \eqref{IV_pf_5_5} follows by induction as claimed.

Finally, note that the existence of a density function and its derivatives up to order $\lceil \eta \rceil-d -1$ follows, as already noted in \eqref{tail_phi_impl_density}.
\end{proof}

\begin{prop}\label{coro_smooth_density_p}
Assume \eqref{c1}-\eqref{c6} and \eqref{c7} with $\eta$ and $\chi$ such that $\limsup\limits_{n\rightarrow\infty} \chi^{n}(\eta) =\infty$. Then, for all $\beta>0$, 
\begin{align*}
|\phi_r(t)|=\bo\left( \|t\|^{-\beta}\right)\text{ for all $r\in[m]$ as } \|t\|\rightarrow\infty.
\end{align*}
In particular, $X_r$ admits a bounded density $f_r\in\mathcal{C}^\infty(\R^d)$ for all $r\in [m]$.
\end{prop}
\begin{proof}
The bound on $|\phi_r(t)|$ follows from Proposition \ref{thm_poly_decay_p} and Lemma \ref{lem_c7_p} below. The second part holds by \eqref{tail_phi_impl_density}.
\end{proof}
\begin{lem}\label{lem_c7_p}
Assume \eqref{c7} and that $|\phi_r(t)|=\bo\left(\|t\|^{-\eta}\right)$ for some $\eta>0$ and all $r\in[m]$.
Then, for all $\beta < \limsup\limits_{n\rightarrow\infty} \chi^{n}(\eta)$,
\begin{align*}
|\phi_r(t)|=\bo\left(\|t\|^{-\beta}\right)\text{ for all $r\in[m]$ as } \|t\|\rightarrow\infty.
\end{align*}
\end{lem}
\begin{proof}
Let $\psi_r(x)=\sup_{\|t\|\geq x} |\phi_r(t)|$ for $r\in [m]$. By assumption there is a constant $K_\eta>0$ such that $\psi_r(x)\leq K_\eta x^{-\eta} \wedge 1$. This implies in combination with \eqref{c7} and Lemma \ref{rec_char_fct} that  $\psi_r(x) \leq K_{\chi(\eta)} x^{-\chi(\eta)}$ with $K_{\chi(\eta)}:=C_{\eta} K_\eta^{\chi(\eta)/\eta}$. Iterating this bound yields the assertion.
\end{proof}

\begin{lem}\label{lem_Schwartz_p}
Let $X$ be a $\R^d$-valued random variable with characteristic function $\phi$. 
Assume $\E[\|X\|^p]<\infty$ for all $p>0$. Moreover, assume for all $\beta>0$ that
\begin{align}\label{eq_lem_Schwartz_p}
|\phi(t)|=\bo\left( \|t\|^{-\beta}\right)\text{ as } \|t\|\rightarrow\infty.
\end{align}
 Then $X$ admits a Schwartz density (see Definition \ref{defSchwartzdens}).
\end{lem}
\begin{proof}
The proof is a straightforward generalization of well known arguments for $d=1$, see, e.g., Fill and Janson for Quicksort \citep{fiJa00}.
First recall that the class of Schwartz functions is preserved under Fourier transformation \eqref{iff_Schwartz_space}. Thus, it is sufficient to show that
$\phi$ is a Schwartz function. 
Note that for all $(\beta_1,\ldots,\beta_d)\in\N_0^d$
\begin{align}\label{pf_main_thm_1}
\left| \frac{\partial^{\beta_1}}{\partial t^{\beta_1}}\cdots \frac{\partial^{\beta_d}}{\partial t^{\beta_d}} \exp\left(i \langle t , X\rangle)\right)\right| \leq \|X\|^{\beta_1+\cdots +\beta_d}
\end{align}
A standard argument based on the Dominated Convergence Theorem reveals that $D^\beta \phi$ exists and is given by
\begin{align}\label{pf_main_thm_2}
D^\beta\phi(t)= \E\left[\frac{\partial^{\beta_1}}{\partial t^{\beta_1}}\cdots \frac{\partial^{\beta_d}}{\partial t^{\beta_d}} \exp\left(i \langle t , X\rangle)\right)\right].
\end{align}
In remains to find constants $d_{\alpha,\beta}>0$ for every $\alpha\in\N_0$ and $\beta=(\beta_1,\ldots,\beta_d)\in\N_0^d$ such that  for all $t\neq 0$
\begin{align}\label{zz_pf_main_last}
 |D^\beta \phi(t)|\leq d_{\alpha,\beta} (\|t\|)^{-\alpha}.
\end{align}
Equations \eqref{pf_main_thm_1} and \eqref{pf_main_thm_2} imply that \eqref{zz_pf_main_last} holds for $d_{0,\beta}=\E[\|X\|^{\beta_1+\cdots +\beta_d}]$ if $\alpha=0$. Moreover, \eqref{eq_lem_Schwartz_p} implies the existence of constants $d_{\alpha,{\bf 0}}$ such that
\eqref{zz_pf_main_last} holds if $\beta={\bf 0}:=(0,\ldots,0)\in\N_0^d$.
 As in \citep[Theorem 2.9]{fiJa00} the remaining cases follow from these cases and 
the next calculus lemma which is a straightforward generalization of the corresponding lemma by Fill and Janson \citep[Lemma 2.10]{fiJa00}.
\end{proof}

\begin{lem}\label{calc_lemma} 
 Let $g:\R^d\rightarrow \C$ be a function such that the partial derivatives $\frac{\partial}{\partial t_j} g$ and $\frac{\partial^2}{\partial t_j^2}g$ exist for some $j\in[d]$. Assume
that $|g(t)|\leq a \|t\|^{-p}$ and $|\frac{\partial^2}{\partial t_j^2} g (t) | \leq b$ for some constants $a,b>0$, $p\geq 0$, and all $t\neq {\bf 0}$. Then $|\frac{\partial}{\partial t_j} g(t)|\leq 2\sqrt{ab} \|t\|^{-p/2}$.
\end{lem}
\begin{proof} We present the full proof to keep the paper self-contained, although we only need to make minor adjustments to the proof of Fill and Janson \citep[Lemma 2.10]{fiJa00}. Note that also $|\frac{\partial^2}{\partial t_j^2} g ({\bf 0})|\leq b$ by Darboux's Theorem, even though the assumption is only stated for $t\neq {\bf 0}$. 

 Fix $t=(t_1,\ldots,t_d)\in\R^d\setminus\{0\}$ and $j\in[d]$. Let $h(x):=g(t_1,\ldots,t_{j-1},x,t_{j+1},\ldots,t_d)$ for $x\in\R$. First consider the case $t_j\geq 0$: Note that for any $y>t_j$
 \begin{align}\label{pf_calc_lem_a}
 \left| \int_{t_j}^y h^\prime(x)\mathrm{d} x\right|= \left| h(y)-h(t_j)\right|\leq |h(y)|+|h(t_j)|\leq 2a \|t\|^{-p},
 \end{align}
 in which the last inequality holds by assumption and by $y>t_j\geq 0$.
 
On the other hand, observe that for $\theta=\mathrm{arg} (h^\prime(t_j))$ and every $x>t_j$:
\begin{align}
\Re\left(\mathrm{e}^{-i\theta} h^\prime (x)\right)
&=\Re\left(\mathrm{e}^{-i\theta} h^\prime (t_j)\right)-\Re\left(\mathrm{e}^{-i\theta} (h^\prime(t_j)-h^\prime (x))\right)\notag\\
&\geq |h^\prime (t_j)|-b (x-t_j),
\end{align}
in which the last inequality holds by the choice of $\theta$, the Mean Values Theorem, and by the bound $|h^{\prime\prime}|\leq b$. Thus, for any $y>t_j$
\begin{align*}
\left| \int_{t_j}^y h^\prime(x)\mathrm{d} x\right|
&=\left| \int_{t_j}^y \mathrm{e}^{-i\theta} h^\prime(x)\mathrm{d} x\right|\\
&\geq \int_{t_j}^y |h^\prime (t_j)|-b (x-t_j) \mathrm{d} x
\\
&=(y-t_j) |h^\prime (t_j)| - \frac {b} 2 (y-t_j)^2.
\end{align*}
Combined with \eqref{pf_calc_lem_a} and the choice $y:=t_j+|h^\prime(t_j)|/b$, this implies $|h^\prime(t_j)|\leq 2 \sqrt{ab} \|t\|^{-p/2}$ as claimed. For $t_j<0$ consider $\tilde{h}(x):=h(-x)$ instead and apply the same bounds.\end{proof}
Finally \cref{coro_smooth_density_p} and \cref{lem_Schwartz_p} imply the main theorem:
\begin{proof}[Proof of \cref{thm_A_conds}] By \cref{coro_smooth_density_p} and \cref{lem_Schwartz_p} it is sufficient to prove \eqref{c1p}-\eqref{c7p} with parameters $\eta>0$ and $\chi:(0,\infty)\rightarrow(0,\infty)$ such that 
\begin{align*}
\lim_{n\rightarrow\infty} \chi^n(\eta) =\infty.
\end{align*}
Recall the assumptions
 \eqref{s_cond_2}-\eqref{s_cond_6}, that is for all $r\in[m]$ and $j\geq 1$:
\begin{flalign}
&\Prob(\amax_r \geq a)=1\quad \text{for some constant $a>0$,}&\label{s_cond_2_p}\tag{A1}\\
&\Prob(\asec_r\leq x) \leq \lambda x^{\nu}\quad \text{ for some $\lambda,\nu>0$ and all $x>0$,}&\label{s_cond_3_p}\tag{A2}\\
&\Prob(\|A_{r,j}^T\|_{\mathrm{op}}\leq 1)=1,&\label{s_cond_4_p}\tag{A3}\\
&\Supp(X_r)\text{ is in general position (see Definition \ref{defSuppGenPos}), }&\label{s_cond_5_p} \tag{A4}\\
&\Prob(N_r( \mathcal{I})\geq 1)>0\text{ for }\mathcal{I}:=(0,1)\subset\R.\tag{A5}&\label{s_cond_6_p}
\end{flalign}
Note that \eqref{s_cond_3_p} in particular implies
\begin{align}\label{asec_0}
 \Prob(\asec_r>0)=1,\quad r\in [m].
 \end{align}
Condition \eqref{c1p} holds by \eqref{s_cond_2_p}. Moreover, \eqref{s_cond_4_p} an \eqref{asec_0} imply $\Prob(N_r((0,1])\geq 2)=1$ and thus \eqref{c2p}. 

 For Condition \eqref{c3p} recall that $\phi_r$ denotes the characteristic function of $X_r$. Note that $\Prob(\langle s, X_r \rangle \in \Z +c)=1$ implies $|\phi_r(2\pi s)|=1$. Hence, it is sufficient to show 
\begin{align}\label{zz_c3} 
 |\phi_r(t)|<1\quad\text{ for all $t\neq {\bf 0}$.}
 \end{align}
First note that \eqref{s_cond_5_p} implies the existence of $\varepsilon_r>0$ such that
 \begin{align}\label{less_1_bound_phi}
 |\phi_r(t)|<1\text{ for all $t\in\R^2$ with }0<\|t\|<\varepsilon_r.
 \end{align}
Details on how \eqref{s_cond_5_p} implies \eqref{less_1_bound_phi}
are stated after the proof (see \cref{lem_phi_less_1}). 
Now let 
\begin{align*}
g_r(x):=\sup_{t:\|t\|=x} |\phi_r(t)|,\qquad g(x):=\max_{r\in[m]} g_r(x),\quad x>0.
\end{align*}
and note that the continuity of $\phi_r$ implies that $g_r(x)=|\phi_r(t_{r,x})|$ for some $t_{r,x}$ with $\|t_{r,x}\|=x$. In particular, by \eqref{less_1_bound_phi}, 
\begin{align*}
g(x)<1\text{ for }x\in(0,\varepsilon),\quad \varepsilon:=\min_{r\in[m]} \varepsilon_r.
\end{align*}
Now suppose \eqref{zz_c3} is false for some $r\in [m]$. Then $g(x)=1$ for some $x\geq \varepsilon$. 
Since $g$ is continuous, the minimum $x_0=\min\{x\geq \varepsilon : g(x)=1 \}$ is attained. Thus we have
\begin{align*}
g(x_0)=1,\qquad g(x)<1\text{ for all }x\in (0,x_0).
\end{align*}
Now choose $r_0\in [m]$ and $t_0\in\R^d$, $\|t_0\|=x_0$, such that
$g(x_0)=|\phi_{r_0}(t_0)|$. Then, by \cref{rec_char_fct},
\begin{align*}
1=|\phi_{r_0}(t_0)|\leq \E\left[\prod_{j=1}^\infty \left| \phi_{\ell_{r_0}(j)} \left(A_{r_0,j}^T t_0\right) \right|\right].
\end{align*}
Hence,
\begin{align*}
\left| \phi_{\ell_{r_0}(j)} \left(A_{r_0,j}^T t_0\right) \right| =1 \text{ a.s. for all }j\geq 1,
\end{align*}
which requires $\|A_{r_0,j}^T t_0\|\notin (0,x_0)$. However, this contradicts \eqref{s_cond_6_p} since
\begin{align*}
\alpha_{r_0,j} x_0 \leq \|A_{r_0,j}^T t_0\|\leq \|A_{r_0,j}^T\|_{\mathrm{op}} x_0.
\end{align*}
Thus $g(x)<1$ for all $x>0$, which yields \eqref{c3p}.

Conditions \eqref{c4p} follows from \eqref{s_cond_3_p} for $\eta<\nu$. Moreover, \eqref{c5p} follows from \eqref{s_cond_2_p}. Condition \eqref{c6p} holds by \eqref{asec_0}.
For Condition \eqref{c7} note that
\begin{align*}
 \E\left[ \prod_{j=1}^{\infty} \left( \alpha_{r,j}^{-\beta} x^{-\beta} \wedge 1\right) \right]
 &\stackrel{\eqref{s_cond_2}}{\leq} 
  a^{-\beta} x^{-\beta} \E[ (\asec_r x)^{-\beta}\wedge 1]\\
  &  \leq a^{-\beta} x^{-\beta}\left( x^{-\beta /2 } + \Prob\left(\asec_r \leq x^{-1/2}\right)\right),
\end{align*}
which yields \eqref{c7} for $\chi(\beta):=\beta+ (\beta\wedge \nu)/2$ by \eqref{s_cond_3}. In particular, $\chi^{n}(\eta)\rightarrow\infty$ for any $\eta>0$ as $n\rightarrow\infty$. Therefore
\cref{coro_smooth_density} yields the existence of bounded densities $f_1,\ldots,f_m\in\mathcal{C}^\infty(\R^d)$ and \cref{lem_Schwartz} implies that $f_r$ is a Schwartz function if all moments of $X_r$ exist. 
\end{proof}

We end the proof section with the missing lemma for \eqref{c3p}. This lemma is a generalization of a standard result for real-valued random variables \cite[XV.1 Lemma 4]{fe71}:
\begin{lem}\label{lem_phi_less_1} Let $X$ be a $\R^d$-valued random variable with characteristic function $\phi$. Let $\mathcal{D}_c^d:=\{v\in\R^d : \|v\|\leq c\}$ for $c>0$. Assume that $\Supp(X)$ is in general position. Then there is an $\varepsilon>0$ such that
\begin{align*}
|\phi( t ) |<1\quad \text{ for all } t\in\mathcal{D}_\varepsilon^d.
\end{align*}
\end{lem}
\begin{proof} Suppose for the sake of contradiction that there is a sequence $(t_n)_{n\geq 1}$ in $\mathcal{D}_\varepsilon^d$ with
\begin{align}\label{lem510_assump_1}
|\phi(t_n)|=1\text{ for all } n\geq 1\quad \text{and} \quad \|t_n\|\rightarrow 0.
\end{align}
Let $c_n=\|t_n\|$ and $\alpha_n=t_n/c_n$. Then \eqref{lem510_assump_1} is equivalent to
\begin{align*}
(a)\quad |\phi(c_n \alpha_n) | =1\text{ for all }n\in\N,\qquad (b)\quad \lim_{n\rightarrow\infty} c_n =0.
\end{align*}
First note that $(a)$ implies that $\phi(c_n\alpha_n)=\exp(i\theta_n)$ for some $\theta_n\in [0,2\pi)$ and therefore
\begin{align}\label{as_lattice_n}
c_n\langle \alpha_n, X\rangle \in 2\pi \Z + \theta_n \quad \text{ a.s.}
\end{align}
Now let $x_0,\ldots,x_d\in\Supp(X)$ be points in general position. Since $x_1-x_0,\ldots,x_d-x_0$ is a basis of $\R^d$, every $\alpha\neq {\bf 0}$ has a $j\in [d]$ such that 
 $\langle \alpha ,x_j-x_0\rangle \neq 0$. In particular, there is a  $j\in[d]$ such that $\langle \alpha_{n},x_j-x_0 \rangle\neq 0$ for infinitely many $n$. Now let $(n_k)_{k\geq 1}$ be a sequence in $\N$ with $n_k\rightarrow\infty$ such that $\langle \alpha_{n_k},x_j-x_0 \rangle\neq 0$ for all $k\geq 1$. Then, by \eqref{as_lattice_n},
\begin{align*}
|\langle \alpha_{n_k},x_j-x_0 \rangle |\geq 2\pi c_{n_k}^{-1} \longrightarrow \infty\quad \text{ as  }k\rightarrow\infty,
\end{align*}
which is a contradiction to $|\langle \alpha_{n_k},x_j-x_0 \rangle |\leq \|\alpha_{n_k}\| \|x_j-x_0\|$ and $\|\alpha_{n_k}\|=1$.
\end{proof}

\section{Conclusion and Remarks}\label{sec:conclusion}
We have seen sufficient conditions for solutions to \eqref{main_rec} that imply the existence of smooth densities. In addition we have seen that the additional assumption of finite moments of any order leads to Schwartz densities for these solutions.

We do not claim that these conditions are sharp in any sense. 
In fact, \eqref{s_cond_2}-\eqref{s_cond_6} are stated for convenience, whereas the slightly weaker conditions \eqref{c1}-\eqref{c7} are  sufficient for the main results. The
 following observations below give some insight on why conditions on $\amax_r$ and $\asec_r$ are reasonable.

{\bf Condition \eqref{s_cond_2}.} In order to see why some kind of lower bound on the largest coefficient in \eqref{main_rec} is required, note that in large continuous time P\'olya urns (with two colors) the following type of distributional equation arises \cite[Proposition 4.2]{ChPouSah11}:
\begin{align*}
X_1&\stackrel d = \mathrm{e}^{-(a-c)\tau}\left(\sum_{j=1}^{a+1} X_1^{(j)} +\sum_{j=1}^b X_2^{(j)}\right),\\
X_2&\stackrel d = \mathrm{e}^{-(a-c)\tau}\left(\sum_{j=1}^{c} X_1^{(j)} +\sum_{j=1}^{d+1} X_2^{(j)}\right),
\end{align*}
with the usual independence assumptions, a standard exponentially distributed $\tau$, and  $a,b,c,d\in\N$.
The system of equations above satisfies \eqref{s_cond_3}-\eqref{s_cond_6} (the solutions are non-degenerate by \cite[Proposition 7.1]{ChPouSah11}). 
It is known \cite[Proposition 7.2]{ChPouSah11} that the limit has a density which explodes at $0$ and thus is not continuous (the densities are infinitely differentiable on $\R\setminus\{0\}$, however). In particular we need to exclude such equations since our methods can only provide continuous densities.

{\bf Condition \eqref{s_cond_3}.} Why do we need lower bounds on the second largest coefficient in \eqref{main_rec}?
The extremal case with only one non-zero coefficient in \eqref{main_rec} is usually called a perpetuity. More precisely, $X$ is called a perpetuity if it satisfies
\begin{align*}
X\stackrel{d} = A X +b
\end{align*}
where $(A,b)$ is independent of $X$. A trivial example for a perpetuity  is a uniformly on $[0,1]$ distributed random variable $U$: 
\begin{align*}
U\stackrel d = \frac 1 2 U + \frac B 2,
\end{align*}
where $B$ and $U$ are independent, and $\Prob(B=0)=\Prob(B=1)=1/2$.
Note that  \eqref{s_cond_2} and \eqref{s_cond_4}-\eqref{s_cond_6} hold in this case. However, the density of $U$ is discontinuous at $0$ and $1$.

{\bf Condition \eqref{s_cond_4}.} Bounding every coefficient in \eqref{main_rec} by $1$ is a convenient condition to deduce \eqref{c1}-\eqref{c7} later on. It is clearly not necessary, but it holds in all 'typical' applications (e.g.~the ones in \cref{sec:motivation}).

{\bf Condition \eqref{s_cond_5}.} Obviously \eqref{s_cond_5} is a necessary condition to obtain a density for $X_r$ (if the condition is violated, then $\Supp(X_r-x)$, $x\in\Supp(X_r)$, is contained in a $d-1$-dimensional subspace of $\R^d$). However, this condition cannot be solely deduced from the coefficients, as already discussed in \cref{rem_s_cond_5}: Many one dimensional examples for \eqref{main_rec} also have a trivial (deterministic) solution, thus we need to assume that the solution is non-degenerate. Also in higher dimensions there are trivial examples where \eqref{s_cond_5} cannot be deduced from the coefficients: Consider, e.g.~the vector $(X,X)^T$ where $X$ is a solution to a one dimensional distributional equation that satisfies \eqref{s_cond_2}-\eqref{s_cond_6}. Then $(X,X)^T$ solves a two dimensional distributional equation with the same coefficients as for $X$. Thus all conditions except \eqref{s_cond_5} still hold for $(X,X)^T$ but $(X,X)^T$ clearly has no density on $\R^2$.

{\bf Condition \eqref{s_cond_6}.} The last condition is made having the 'typical' applications in mind. Like \eqref{s_cond_4}, this condition is convenient to deduce \eqref{c1}-\eqref{c7} later on. It is no real restriction to most applications arising in the contraction method: 
In order for the limit map to be a contraction, one needs to make an assumption like $\sum_j \E[\|A_{r,j}\|_{\mathrm{op}}^p]^{1/p}<1$ for some $p\geq 1$. In particular, a requirement for the contraction method to work is
	\begin{align}\label{conclude_contraction}
\Prob\left(\bigcap_{j\geq 1}\left\{ \|A_{r,j}\|_{\mathrm{op}}<1\right\}\right)>0.
\end{align}
Since \eqref{s_cond_2} and \eqref{conclude_contraction} imply \eqref{s_cond_6}, no additional restrictions to 'typical' applications are made by \eqref{s_cond_6}.

{\bf Some further directions and open problems.}
It remains an interesting open problem to derive other characteristics of $\mathcal{L}(X_r)$ from \eqref{main_rec}. In particular, as an addition to this article, conditions on \eqref{main_rec} that imply finiteness of all moments are of special interest. If $m=1$ and $d=1$, R\"osler \citep[Theorem 6]{Roe92} presents an approach (if \eqref{main_rec} is a fixed point equation of a contraction mapping) that leads to $\E[\exp(\lambda X_r)]<\infty$ for all $\lambda$ in some open neighborhood of $0$. A generalization of this approach could lead to conditions in the general case ($m,d\geq 1$) for finite exponential moments. 

Moreover, bounds on $\Prob(X_r\geq x)$ and $\Prob(X_r\leq -x)$ for (large) $x>0$ are desirable, especially for applications in computer science. Once again, a generalization of R\"osler's approach in combination with Markov's inequality could be used to get exponential tail bounds in some applications.
 However, at least for Quicksort the tails of the limit distribution decay even faster:

Let $X$ be the Quicksort limit in \eqref{DE_Quicksort}.
 Knessl and Szpankowski \citep{KneSzpa99} have found with non-rigorous methods (based on several unproven assumptions) some constants $c_1,x_2,\gamma>0$ such that, as $x\rightarrow\infty$,
\begin{align*}
\Prob(X\leq -x)\sim c_1\exp\left(-c_2\exp(\gamma x)\right),\qquad \Prob(X\geq x)=\exp\left(-x\ln x - x \ln\ln x +\bo(x)\right).
\end{align*}
Janson \citep{Ja15} showed $\Prob(X\leq -x)\leq \exp(-x^2/5)$ and $\Prob(X\geq x)\leq \exp(-x\ln x +(1+\ln 2) x)$ rigorously for all sufficiently large $x$. \\

\noindent
{\bf Acknowledgements.} The author thanks Ralph Neininger and Henning Sulzbach for valuable comments on the topic and related literature. The author also thanks the two unknown referees for their careful reading and valuable remarks.

\bibliographystyle{plainnat}
\bibliography{library}{}

\end{document}